\documentclass{article}

\usepackage{amssymb,amsmath,amsthm,amsfonts,enumerate}
\usepackage{color, graphicx,multicol}

\usepackage{enumerate}
\usepackage{authblk}

\newcommand{\intR}{\int_{-\infty}^{\infty}}

\allowdisplaybreaks \numberwithin{equation}{section}

\theoremstyle{plain}
\newtheorem{thm}{Theorem}[section]
\newtheorem{cor}[thm]{Corollary}
\newtheorem{lem}[thm]{Lemma}

\newtheorem{defn}[thm]{Definition}

\def\intR{\int_{-\infty}^\infty}

\renewcommand\Re{\operatorname{\mathfrak{Re}}}
\renewcommand\Im{\operatorname{\mathfrak{Im}}}

\DeclareMathOperator{\Res}{Res}

\title{\bf A uniqueness result for 2-soliton solutions of the KdV equation}

\author[1]{John P. Albert}
\author[2]{Nghiem V. Nguyen}
\affil[1]{Department of Mathematics, University of Oklahoma, Norman OK 73019, \texttt{jalbert@ou.edu}}
\affil[2]{ Department of Mathematics and Statistics, Utah State University, Logan UT 84322, \texttt{nghiem.nguyen@usu.edu}}

\begin{document}
\maketitle

\begin{abstract}
Multisoliton solutions of the KdV equation
satisfy nonlinear ordinary differential equations which are known as
stationary equations for the KdV hierarchy, or sometimes as Lax-Novikov equations.  An interesting feature of these equations,
known since the 1970's, is that they can be explicitly integrated, by virtue of being finite-dimensional
completely integrable Hamiltonian systems. Here we use the integration theory to investigate the question
of whether the multisoliton solutions are the only nonsingular solutions of these ordinary differential equations
which vanish at infinity.  In particular we prove that this is indeed the case for $2$-soliton solutions of the
fourth-order stationary equation.
\end{abstract}

\section{Introduction}
\label{sec:intro}
\renewcommand{\theequation}{\arabic{section}.\arabic{equation}}
\setcounter{section}{1} \setcounter{equation}{0}

\bigskip

The Korteweg-de Vries (or KdV) equation,
\begin{equation}
u_t=\frac14(u_{xxx}+6uu_x),
\label{KdV}
\end{equation}
was first derived in the
1800's as a model for long, weakly nonlinear one-dimensional water
waves (\cite{KdV}, see also equation (283 {\it bis}) on p.\ 360 of \cite{Bou}).  It was not
until the 1960's, however, that the striking discovery was made
that the equation has particle-like solutions known as solitons,
whose interactions with each other are described by explicit
multisoliton solutions \cite{GGKM, H}.

It is well-known that the profiles of multisoliton solutions, which are smooth functions that
vanish rapidly at infinity, are critical points for
variational problems associated with conserved functionals of KdV (see, e.g., \cite{MS}).
By virtue of this property, the profiles are solutions of Lagrange
multiplier equations, which take the form of nonlinear ordinary
differential equations, sometimes known as Lax-Novikov equations, or as the equations
for stationary solutions of a family of time-dependent equations known as the KdV hierarchy (see below
for details).  In this paper we investigate the problem
of establishing a converse to this statement:  is it true that if a
solution to a stationary equation for the KdV hierarchy is, together with enough of its
derivatives, square integrable on the real line, then must it be a
profile of a multisoliton solution?

For the case of the KdV equation itself (the first equation in the hierarchy), it is an elementary exercise to prove that the only stationary solutions in $L^2$ are the well-known solitary-wave solutions.  Here we give a proof that the answer is also affirmative for the case of the fourth-order stationary equation for the second, fifth-order, equation in the
KdV hierarchy  (see Theorem \ref{2solthm} below).  Much of our
proof easily generalizes to the other stationary equations for the
hierarchy, but some work remains to be done to complete the proof in the
general case.

Our proof proceeds by integrating the stationary equations, using the method developed in the
pioneering work of  Dubrovin \cite{Du}, Its and Matveev \cite{IM}, Lax \cite{L}, and Novikov \cite{N} on solutions of the periodic KdV equation.  An early survey of the work of these authors is \cite{DMN}, and more recent treatments are \cite{BEIM} and \cite{GH}.   For a lively historical account of the development of the subject, we refer the reader to \cite{M1}, in which it is noted that elements of the theory, including in particular  equation \eqref{eqR2}, can be traced back at least as far the work of Drach \cite{Dr} in 1919. Here we follow the approach of Gel'fand and Dickey, which first appeared in \cite{GD},
and has received a nice expository treatment in Chapter 12 of Dickey's book \cite{Di}.  In this approach, the stationary equations, which have the structure of completely integrable Hamiltonian systems, are rewritten in action-angle variables, which reduces them to an easily integrable set of equations (see \eqref{sys2} below) first obtained by Dubrovin in \cite{Du}. (We remark that each stationary equation is a finite-dimensional completely integrable Hamiltonian system in the classical sense; unlike the time-dependent KdV equations which are in some sense \cite{FZ, G} infinite-dimensional completely integrable Hamiltonian systems.) Integrating Dubrovin's equations shows that every smooth solution of the stationary equations must be expressible in the form given below in \eqref{itsmat}, which is known as the Its-Matveev formula \cite{IM}.  It turns out that this part of the proof is valid for all stationary equations for the KdV hierarchy.   We then conclude by determining which solutions of the Its-Matveev formula are nonsingular.  The latter step we have so far only completed for the second stationary equation in the hierarchy: that is, for the equation for 2-solitons.

We emphasize that our interest here is not in constructing solutions of the stationary equations; all the solutions appearing in this paper are already well-known (see, for example, \cite{M2}).  Rather, our focus is on showing that
a corollary of the method used to integrate these equations is that the $N$-solution solutions are the {\it only} solutions with finite energy, at least in the case $N=2$. Also, we have made an effort to give a self-contained presentation, which in fact relies entirely on elementary calculations.

The result we prove here has consequences for the stability theory
of KdV multisolitons.  As we show in a forthcoming paper, it can
be used to show that two-soliton solutions of KdV are global
minimizers for the third invariant of the KdV equation, subject to
the constraint that the first two invariants be held constant.
This in turn establishes the stability of two-soliton solutions,
thus providing an alternative proof to that appearing in
\cite{MS}.

We remark that in order to be useful for the stability theory, it is important that our uniqueness result
make no assumption on the values of the parameters $d_i$ appearing in equation \eqref{stathier2}.
This requirement influenced our choice of method of proof.  An alternate method we considered was to 
proceed by an argument which counts the dimensions of the stable and unstable manifolds of \eqref{stathier2}
at the origin in phase space.  Indeed, if one assumes in advance
that $d_3$ and $d_5$ are such that equation \eqref{c1c2} has distinct positive roots, then this method does give,
after some work, that the well-known 2-soliton solutions are the only homoclinic solutions of \eqref{stathier2}. Such an argument, however, becomes much
more complicated for other choices of $d_i$, partly because center manifolds of dimension up to 4 can appear.  For this
reason, we have found it better to proceed by direct integration of the equation instead.

The plan of the paper is as follows. In Sections
\ref{sec:reviewKdv} and \ref{sec:Nsoliton}, for the reader's
convenience and to set notation, we review some of the basic
properties of multisoliton solutions.  In Section
\ref{sec:reviewKdv} we introduce the equations of the KdV
hierarchy, and their associated stationary equations.  (Here
``stationary'' means ``time-independent'': stationary equations
are equations for time-independent solutions of the KdV hierarchy.
Coincidentally, they are also equations for stationary points of
variational problems.)  In Section \ref{sec:Nsoliton} we define
the $N$-soliton solutions of the KdV hierarchy, and give a proof
of the well-known fact that their profiles are actually solutions
of stationary equations. Section \ref{sec:1solconverse} prepares
for the main result by treating the elementary case of stationary
solutions of the KdV equation itself. In Section
\ref{sec:2solconverse} we prove the main result, which is that for
the stationary equation for the fifth-order equation in the KdV
hierarchy, the only $H^2$ solutions are $1$-soliton and
$2$-soliton profiles.  A concluding section discusses the question
of how to generalize the result to higher equations in the hierarchy and
$N$-solitons for $N> 2$.

\section{The KdV hierarchy} \label{sec:reviewKdv}

We review here the definition of the KdV hierarchy, following the
treatment of chapter 1 of \cite{Di}.

Let $\mathcal A$ denote the differential algebra over $\mathbf C$
of formal polynomials in $u$ and the derivatives  of $u$.  That
is, elements of $\mathcal A$ are polynomials with complex
coefficients in the symbols $u$, $u'$, $u''$, etc.; and elements
of $\mathcal A$ can be acted on by a derivation $\xi$, a linear
operator on $\mathcal A$ which obeys the Leibniz product rule, and
takes $u$ to $u'$, $u'$ to $u''$, etc.  We adopt the convention
that primes also denote the action of $\xi$ on any element of
$\mathcal A$. Thus the expressions $a'$ and $\xi a$ are synonymous
for $a \in \mathcal A$. Later we will substitute actual functions
of $x$ for $u$, and then $\xi$ will correspond to the operation of
differentiation with respect to $x$, so that $u'$, $u''$, etc.,
will denote the derivatives of these functions with respect to $x$
in the usual sense.

If $M$ is an integer, we define a pseudo-differential operator of
order $M$ to be a formal sum
\begin{equation}
X=\sum_{i=-\infty}^M a_i\partial^i, \label{psdiff}
\end{equation}
where
$a_i \in \mathcal A$ for each $i$.  Clearly the set $\mathcal P$
of all pseudo-differential operators has a natural module
structure over the ring $\mathcal A$.  We can also make $\mathcal
P$ into an algebra by first defining, for each integer $k$ and
each $a \in \mathcal A$, the product $\partial^k a$ as
\begin{equation*}
\partial^k a = a \partial^k + {k \choose 1} a' \partial^{k-1}+ {k
\choose 2} a'' \partial^{k-2}+ \dots,
\end{equation*}
where
\begin{equation*}
{k \choose i} = \frac{k(k-1)\cdots(k-i+1)}{i!};
\end{equation*}
and then extending this multiplication operation to all of
$\mathcal P$ in the natural way:
\begin{equation*}
\begin{aligned}
\left(\sum_{i=-\infty}^M a_i\partial^i\right)
\left(\sum_{j=-\infty}^N b_j\partial^j\right)&=
\sum_{i=-\infty}^M
\sum_{j=-\infty}^N a_i(\partial^i b_j)\partial^j \\
&=\sum_{i=-\infty}^M
\sum_{j=-\infty}^N \sum_{l=0}^\infty a_i {i \choose l}(\xi^lb_j)\partial^{i+j-l}.
\end{aligned}
\end{equation*}
The last sum in the preceding equation is well-defined in
$\mathcal P$ because each value of $i+j-l$ occurs for only
finitely many values of the indices $i,j,l$.   It can be checked
that, with this definition of multiplication, $\mathcal P$ is an
associative algebra with derivation $\partial$. Interestingly,
this algebra $\mathcal P$ was studied by Schur in \cite{Sch}, many
years before its utility for the theory of integrable systems was
discovered.

In particular we will have $\partial \partial^{-1}= 1$ in
$\mathcal P$.  More generally, suppose $X$ is given by
\eqref{psdiff} with $a_M = 1$.  Then $X$ has a multiplicative
inverse $X^{-1}$ in $\mathcal P$; this may be verified by first
observing that the order of $X^{-1}$ must be $-M$ and then using
the equation $X X^{-1}=1$ to solve recursively for the
coefficients $b_i$ of $X^{-1}=\sum_{i=-\infty}^{-M}b_i\partial^i$.
Carrying out this process, one finds that the $b_i$ are
polynomials in the $a_i$ and their derivatives.  Similarly, there
exists $Y \in \mathcal P$ such that $Y^m = X$, as may be proved by
observing that $Y$ must be of order 1, and using the equation $Y^m
= X$ to solve recursively for the coefficients $c_i$ of
$Y=\sum_{i=-\infty}^1 c_i\partial^i$. These coefficients will be
uniquely determined if we specify that $c_1=1$, and in that case
the operator $Y$ so obtained will be denoted by $X^{1/m}$. (All
other solutions of $Y^m=1$ are of the form $\alpha X^{1/m}$ where
$\alpha$ is an $m$th root of unity.) For $k \in \mathbf Z$ we then
define $X^{k/m}$ to be the $k$th power of $X^{1/m}$. Since $X$ is
an integer power of $X^{1/m}$ it follows immediately that $X$ and
$X^{1/m}$ commute, and hence so do $X$ and $X^{k/m}$.

If in \eqref{psdiff} we have $a_i=0$ for all $i < 0$, then we say
that $X$ is a differential operator; obviously the product and sum
of any two differential operators is again a differential
operator. For general $X \in \mathcal P$, the differential part of
$X$, denoted by $X_+$, is defined to be the differential operator
obtained by omitting all the terms from $X$ which contain
$\partial^i$ with negative $i$. We also define $X_-$ to be
$X-X_+$. As usual, we define the commutator $[X_1,X_2]$ of two
elements of $\mathcal P$ by $[X_1,X_2]=X_1 X_2 - X_2 X_1$. Also,
if $X$ is given by \eqref{psdiff}, it will be useful to define the
residue of $X$, $\Res X$, to equal $a_{-1}$.  That is, $\Res X \in
\mathcal A$ is the
 coefficient of $\partial^{-1}$ in the expansion of $X$.  Finally,
 for $X$ as in \eqref{psdiff} we define $\sigma_2(X) \in  \mathcal P$
 by
\begin{equation}
\sigma_2(X)=a_{-1}\ \partial^{-1} + a_{-2}\ \partial^{-2}.
\label{defsig2}
\end{equation}

The Korteweg-de Vries hierarchy can be defined in terms of fractional powers of the differential operator $L$ given
by
\begin{equation}
L=\partial^2 + u. \label{L}
\end{equation}
From the above considerations, $L^{(2k+1)/2}$ is well-defined as
an element of $\mathcal P$ for each nonnegative integer $k$. When
we take its differential part, we obtain the operator
$\left(L^{(2k+1)/2}\right)_+$, which has the following important
property.

\begin{lem} The commutator $[\left(L^{(2k+1)/2}\right)_+,L]$ is a differential operator of order
0; that is, a polynomial in $u$ and its derivatives.  In fact, it is given by the equation
\begin{equation}
\left[\left(L^{(2k+1)/2}\right)_+,L\right]=2\left(\Res L^{(2k+1)/2}\right)'.
\label{commeqres}
\end{equation}
\label{kdvrhs}
\end{lem}

\begin{proof} As the commutator of two differential operators,
$[\left(L^{(2k+1)/2}\right)_+,L]$ is a differential operator.  Now
\begin{equation*}
[\left(L^{(2k+1)/2}\right)_+,L]=[L^{(2k+1)/2},L]-[\left(L^{(2k+1)/2}\right)_-,L],
\end{equation*}
and, as noted above, $L^{(2k+1)/2}$ commutes with $L$, so
\begin{equation}
[\left(L^{(2k+1)/2}\right)_+,L]=-[\left(L^{(2k+1)/2}\right)_-,L].
\label{comm}
\end{equation}
Observe that, in general, the commutator of an operator of order
$M_1$ and an operator of order $M_2$ has order $M_1 + M_2 -1$.
Since the right hand side of \eqref{comm} is a commutator of an
operator of order $-1$ and an operator of order 2, it therefore
has order 0.

Once it is established that both sides of \eqref{comm} are equal to a differential operator of order 0,
the identity \eqref{commeqres} is easily obtained by computing the term of order 0 in the expansion
of $-[\left(L^{(2k+1)/2}\right)_-,L]$.
\end{proof}

The Korteweg-de Vries hierarchy is a set of partial differential
equations, indexed by the natural numbers $k=0,1,2,3,\dots$, for
functions $u(x,t_{2k+1})$ of two real variables $x$ and $t_{2k+1}$. The $k$th
equation in the hierarchy is defined as
\begin{equation}
u_{t_{2k+1}} = 2\left(\Res L^{(2k+1)/2}\right)'. \label{KdV-k1}
\end{equation}
Here the subscripted $t_{2k+1}$ denotes the derivative with respect to
$t_{2k+1}$.  Starting with $k=0$, the first three in the hierarchy are given by:
\begin{equation}
\begin{aligned}
u_{t_1}&=u',\\
u_{t_3}&=\frac14 (u'''+6uu'),\\
u_{t_5}&=\frac{1}{16}(u'''''+10uu'''+20u'u''+30u^2u'). \end{aligned}
\label{examples}
\end{equation}
The second equation in \eqref{examples} is the KdV
equation \eqref{KdV}.

This definition of the hierarchy is due to Gelfand and Dickey, and
leads to simple formulations and proofs of many properties of
these equations, including the fact that they define commuting flows, which were formerly proved by more unwieldy
methods. Also, natural modifications of the definition lead
readily to more general hierarchies of equations (today called
Gelfand-Dickey hierarchies), of which the Korteweg-de Vries
hierarchy is just one, and which share many of the interesting
integrability properties of the Korteweg-de Vries hierarchy
\cite{Di, Di2}.

An important feature of the KdV hierarchy \eqref{KdV-k1} is that the differential polynomials which appear
on the right-hand side satisfy a simple recurrence relation.  Following the notation of Chapter 12 of \cite{Di}, let us define,
for $k=0,1,2,\dots,$
\begin{equation*}
R_{2k+1}=\frac{(-1)^k}{2}\Res L^{(2k-1)/2},
\end{equation*}
so that the KdV hierarchy takes the form
\begin{equation}
u_{t_{2k+1}} = 4(-1)^{k+1} R_{2k+3}'. \label{kdvhier}
\end{equation}

\begin{lem}
The differential polynomials $R_{2k+1}$ satisfy the recurrence relation
\begin{equation}
R_{2k+1}'''+4uR_{2k+1}'+2u'R_{2k+1} = -4R_{2k+3}', \label{Rrecur}
\end{equation}
for $k=0,1,2,\dots$,
with initial condition $R_1 = 1/2$.
\label{recrel}
\end{lem}

\begin{proof}
The proof of this lemma is essentially an exercise on the material
in Section 1.7 of \cite{Di}, but for the reader's convenience we
indicate the details here.

Let $\mathcal C$ denote the set of all formal Laurent series in
$z$ of the form  $\displaystyle \sum_{r=-\infty}^\infty X_r z^r$, where $X_r \in
\mathcal P$ for $r \in \mathbf Z$.  Then $\mathcal C$ inherits an
operation of addition from $\mathcal P$, and if $S$ and $T$ are in
$\mathcal C$ and all but finitely many of the coefficients in $T$
are zero, then the products $ST$ and $TS$ are defined in $\mathcal
C$ by the usual term-by-term multiplication of series. Also, for
$S=\sum_{r=-\infty}^\infty X_r z^r$ in $\mathcal C$ we define
$S_+=\sum_{r=-\infty}^\infty (X_r)_+ z^r$,
$S_-=\sum_{r=-\infty}^\infty (X_r)_- z^r$,
$\Res(S)=\sum_{r=-\infty}^\infty (\Res X_r) z^r$, and
$\sigma_2(S)=\sum_{r=-\infty}^\infty \sigma(X_r) z^r$, where $\sigma_2$ is the operator defined in \eqref{defsig2}.

Let $\widehat L = L - z^2$, and define the map $H:\mathcal C \to
\mathcal C$ by
\begin{equation*}
H(X)=(\widehat LX)_+\widehat L -\widehat L(X\widehat L)_+.
\end{equation*}
(Dickey \cite{Di} calls $H$ the Adler map, as it was introduced in
section 4 of \cite{Ad}.) Since $(\widehat LX)\widehat L =\widehat L(X\widehat L)$, it follows that $H(X)=(\widehat
LX)_-\widehat L -\widehat L(X\widehat L)_-$ for all $X$ in
$\mathcal C$. Moreover, from the definition of $H$ and the fact
that $L$ is a differential operator of order 2, one sees easily
that $H(X)=H(\sigma_2(X))$ for all $X$ in $\mathcal C$.

Define
\begin{equation*}
T=\sum_{r=-\infty}^\infty L^{r/2}z^{-r-4}.
\end{equation*}
Clearly $\widehat L T = T\widehat L = 0$, so $H(T)=0$, and
hence also $H(\sigma_2(T))=0$. On the other hand, by observing that $\sigma_2(L^{2k})=0$ for all nonnegative integers $k$,
$\sigma_2(L^{r/2})=0$ for all integers $r \le -3$, and $\sigma_2(L^{-1})=\partial^{-2}$, we can write $\sigma_2(T)$ as
\begin{equation}
\sigma_2(T)=R\partial^{-1}+ \tilde R \partial^{-2},
\label{sig2T}
\end{equation}
where $R$ and $\tilde R$ are in $\mathcal C$ and
\begin{equation}
R=\Res \sigma_2(T)=\sum_{k=0}^\infty 2(-1)^k R_{2k+1}\ z^{-2k-3}.
\label{defR}
\end{equation}  Substituting \eqref{sig2T} into the equation $H(\sigma_2(T))=0$, we find
after a computation that
\begin{equation}
0=H(\sigma_2(T))=(-\tilde R''+2uR'+u'R-2z^2R')-(R''+2\tilde R')\partial,
\end{equation}
and therefore $\tilde R = -\frac12 R'$ and
\begin{equation}
\frac12 R''' + 2uR' + u'R = 2z^2R'.
\label{eqR2}
\end{equation}
Then substituting \eqref{defR} into \eqref{eqR2} gives \eqref{Rrecur} for $k=0,1,2,\dots$.  Finally, we can verify that
$R_1=1/2$ by directly computing $R_1 = \frac12 \Res L^{-1/2}$.

\end{proof}

Using the recurrence relation in Lemma \ref{recrel}, we find, for example, that the first few terms in the sequence $\{R_{2k+1}\}$
are

\begin{equation}
\begin{aligned}
R_1 &= 1/2,\\
R_3 &= (-1/4)u,\\
R_5 &= (1/16)(u''+3u^2),\\
R_7 & = (-1/64)(u''''+5u'^2+10uu''+10u^3).
\end{aligned}
\label{R1357}
\end{equation}

Of particular interest are time-independent or stationary
solutions of \eqref{kdvhier}.  If $u$ is a such a solution, then
$u$ satisfies \eqref{kdvhier} with $u_t = 0$, and hence
integration gives that $u$ satisfies the equation $R_{2k+3}=d$,
where $d$ is a constant, independent of $x$ and $t$.  Letting
$d_1=2d$, we can rewrite this equation in the form
\begin{equation*}
d_1R_1-R_{2k+3}=0.
\end{equation*}
More generally, we can view any solution of the equation
\begin{equation}
d_1R_1 + d_3R_3 + d_5R_5 + \dots + d_{2N+3}R_{2N+3}=0
\label{stathier}
\end{equation}
as a stationary solution of the equation
\begin{equation*}
u_t=d_3R_3'+d_5R_5'+\dots+d_{2N+3}R_{2N+3}',
\end{equation*}
which itself can be considered to be an equation in the KdV
hierarchy.   For this reason, following \cite{Di}, we refer to
equations \eqref{stathier} as the stationary equations of the KdV
hierarchy.  (They are also
sometimes called Lax-Novikov equations.)

Equation \eqref{stathier} is an ordinary
differential equation of order $2N$, and can therefore be
rewritten as a first-order system in phase space $\mathbf R^{2N}$.
It turns out that this system is of Hamiltonian form, and in fact
is completely integrable in the sense that it has $N$ independent
integrals in involution with each other.  In general, Liouville's
method provides a technique for actually integrating completely
integrable systems: that is, for explicitly finding the
transformation from coordinates of phase space to action-angle
variables.  However, this integration involves solving a system of
first-order partial differential equations.  For the system
\eqref{stathier}, Dubrovin \cite{Du} introduced a change of
variables under which this system of PDE's has a simple form and
is trivially solvable.  This is the change of variables we use below in Section
\ref{sec:2solconverse}.

\section{N-soliton profiles}
\label{sec:Nsoliton}

Another key aspect of the KdV hierarchy is that the flows which it defines all commute with each other, at least formally.
More precisely, one can check that the equations in \eqref{kdvhier}
have the formal structure of Hamiltonian equations with respect to a certain symplectic form, and are all in involution with each other
with respect to this form (see, for example, chapters 1 through 4 of \cite{Di}).  This suggests the following.  Assume that a function class $\mathcal S$ has been defined
such that for each $k \in \mathbf N$, the initial-value problem for equation \eqref{kdvhier} is well-posed on $\mathcal S$,
and let $S(t_{2k+1})$ be the solution map for this problem, which to each $\psi \in \mathcal S$ assigns the function
$S(t_{2k+1})[\psi]=u(\cdot,t_{2k+1}) \in \mathcal S$, where $u$ is the solution of \eqref{kdvhier} with initial data $u(x,0)=\psi(x)$.  Then
in light of the formal structure mentioned above, one would expect that the solution operators $S(t_{2k+1})$ and $S(t_{2l+1})$ commute with each other
as mappings on $\mathcal S$. Hence, for each $\psi \in \mathcal S$ and each $l \in \mathbf N$, one should be able to
define a simultaneous solution $u(x,t_1,t_3,t_5,\dots,t_{2l+1})$ to all
of the first $l$ equations in the hierarchy  by setting
\begin{equation*}
u(x,t_1,t_3,\dots,t_{2l+1})=S(t_1)S(t_3)\cdots S(t_{2l+1})\psi.
\end{equation*}

This formal analysis, however, does not lead easily to concrete results about general solutions
of the KdV hierarchy.  For this reason there has historically been great interest in constructing and elucidating the
structure of explicit solutions.  In this section we review the definition and basic properties of an important class of such solutions, the $N$-soliton solutions.

To begin the construction of $N$-soliton solutions, let $N \in
\mathbf N$, and for $1 \le j \le N$ define the functions
 \begin{equation}
y_j(x)=e^{\alpha_j x}+a_je^{-\alpha_j x}, \label{defyix}
\end{equation}
where $\alpha_j$ and
$a_j$ are  complex numbers  satisfying

\begin{equation}
\begin{aligned}
&\text{(i) for all $j \in \{1,\dots,N\}$, $\alpha_j \ne 0$ and $a_j
\ne 0$,} \\
&\text{(ii) for all $j, k \in \{1,\dots,N\}$, if $j < k$ then
$\alpha_j \ne \alpha_k$ and
 $0 \le \Re \alpha_j \le \Re \alpha_k$.}
\end{aligned}
\label{paramcond}
\end{equation}
We will use $D(y_1,\dots,y_N)$ to denote
the Wronskian of $y_1$, \dots $y_N$:
\begin{equation}
D(y_1,\dots,y_N) =\left|
\begin{matrix}
y_1 & \dots &y_N  \\
y_1' & \dots &y_N' \\
\dots & \dots & \dots \\
y_1^{(N-1)} & \dots & y_N^{(N-1)}  \\
\end{matrix}
\right|.
 \label{Delta}
\end{equation}

Next we will construct an operator of the form \eqref{L} from the
$y_j$, using a technique known as the ``dressing
method'' \cite{Di}.  First, on any interval $I$ where $D \ne 0$,
we define a differential operator $\phi$ of order $N$ by
\begin{equation}
\phi=\frac{1}{D}\left|
\begin{matrix}
y_1 & \dots &y_N & 1 \\
y_1' & \dots &y_N' & \partial \\
\dots & \dots & \dots & \dots \\
y_1^{(N-1)} & \dots & y_N^{(N-1)} & \partial^{N-1} \\
y_1^{(N)} & \dots & y_N^{(N)} & \partial^N
\end{matrix}
\right|.
 \label{defphi}
\end{equation}
Here it is understood that the determinant in \eqref{defphi} is
to be expanded along the final column, multiplying each
operator $\partial^i$ by its corresponding cofactor on the
left.  In other words,
\begin{equation}
\phi=\partial^N +
W_{N-1}\partial^{N-1}+W_{N-2}\partial^{N-2}+\dots+W_1
\partial + W_0, \label{expandphi}
\end{equation}
 where for $i=1,\dots,N-1$,
\begin{equation}
W_i=\frac{(-1)^{N+2+i}}{D}\left|
\begin{matrix}
y_1 & \dots &y_N\\
y_1' & \dots &y_N' \\
\dots & \dots & \dots  \\
y_1^{(i-1)} & \dots & y_N^{(i-1)}\\
y_1^{(i+1)} & \dots & y_N^{(i+1)}\\
\dots & \dots & \dots  \\
y_1^{(N)} & \dots & y_N^{(N)}
\end{matrix}
\right|,
 \label{defWi}
\end{equation}
and
\begin{equation}
W_0=\frac{(-1)^{N+2}}{D}\left|
\begin{matrix}
y_1' & \dots &y_N'\\
y_1'' & \dots &y_N'' \\
\dots & \dots & \dots  \\
y_1^{(N)} & \dots & y_N^{(N)}
\end{matrix}
\right|.
 \label{defW0}
\end{equation}


Next, we slightly generalize the notion of pseudo-differential operator defined in
Section \ref{sec:reviewKdv} to include formal sums of type \eqref{psdiff} in which the $a_i$
are no longer differential polynomials in a single variable $u$, but now are rational functions
of the $n$ symbols $y_1, y_2, \dots, y_n$ and their formal derivatives $y_1', y_1'', y_2', y_2''$, etc.
(forgetting for the moment that $y_1, y_2, \dots, y_n$ are actually functions of $x$).  The definitions of the multiplication and inverse
operations on pseudo-differential operators given in Section \ref{sec:reviewKdv} remain unchanged
for this larger algebra.  Thus, $\phi$ as defined in \eqref{defphi} has a formal
inverse $\phi^{-1}$, which is a pseudo-differential operator whose
coefficients are rational functions of $y_i$ and their derivatives, expressible as polynomials in $W_j$ and their
derivatives.   We now define $L$ as the formal pseudo-differential
operator given by
\begin{equation}
L=\phi \partial^2 \phi^{-1}. \label{genL}
\end{equation}
 
\begin{lem}
The differential part of $L$ is
\begin{equation}
L_+ = \partial^2 -2W_{N-1}'. \label{Lpos}
\end{equation}
\label{Lposlem}
\end{lem}

\begin{proof}
First observe that
\begin{equation}
L=(\phi\partial^{-N})\partial^2(\phi
\partial^{-N})^{-1}.
\label{obs1}
\end{equation}
Now we can write
\begin{equation}
\phi
\partial^{-N}=1+W_{N-1}\partial^{-1}+W_{N-2}\partial^{-2}+O(\partial^{-3}),
\label{obs2}
\end{equation}
where ``$O(\partial^{-3})$'' denotes terms containing $\partial^j$
with $j \le -3$.  Also, a computation shows that
\begin{equation}
(\phi
\partial^{-N})^{-1}=1-W_{N-1}\partial^{-1} +
\left(W_{N-1}^2-W_{N-2}\right)\partial^{-2}+O(\partial^{-3}).
\label{obs3}
\end{equation}
Equation \eqref{Lpos} then follows easily by inserting
\eqref{obs2} and \eqref{obs3} into \eqref{obs1} and carrying out the multiplication to determine
the terms of nonnegative order.
\end{proof}

\begin{lem}
Define $W_{-1}=0$.  Then
\begin{equation}
L_-\phi =-\sum_{j=0}^{N-1}
\left(W_j''+2W_{j-1}'-2W_{N-1}'W_j\right)\partial^j.
\label{Lmincoeff}
\end{equation} \label{Lminlem}
\end{lem}

\begin{proof}
Since $L_-=L-L_+$, we have from \eqref{genL} that
\begin{equation}
L_-\phi=\phi\partial^2-L_+\phi. \label{Dickey}
\end{equation}
The desired result follows by substituting \eqref{expandphi} and
\eqref{Lpos} into the right-hand side, and carrying out the
multiplications.
\end{proof}

So far, in discussing $\phi$ and $L$, we have considered them only
as formal pseudo-differential operators with coefficients that are
rational functions in the symbols $y_i$, $y_i'$, $y_i''$,\dots.
Now, however, we wish to ``remember'' the fact that these
coefficients are specific functions of $x$.  To this end we first
observe that by \eqref{paramcond} the functions $y_1,\dots,y_N$
are analytic and linearly independent.  Therefore, by a theorem of
Peano (see \cite{Boc}), their Wronskian $D$ cannot vanish
identically on any open interval in $\mathbf R$.  In particular,
by continuity there exists an open interval $I$ on $\mathbf R$
such that $D(x) \ne 0$ for all $x \in I$.  Therefore, on $I$ the
right-hand-side of \eqref{defphi} defines a linear differential
operator with smooth coefficients $W_i$ for $i=1,\dots,N-1$.  To
emphasize the distinction between the formal operator $\phi$ and
its concrete realization, we introduce the notation $r(\phi)$ for
the differential operator with smooth coefficients on $I$ obtained
by remembering that the $y_i$ are certain functions of $x$.

More generally, if $X$ is any pseudo-differential operator
whose coefficients are formal polynomials in $W_i$ and their derivatives,
we define $r(X)$ to be the operator obtained by remembering that
the coefficients of $X$ are actually smooth functions of $x$ on
$I$.  Thus $r$ defines an algebra homomorphism from $\mathcal P$ to the
 to the algebra of pseudo-differential operators
with coefficients that are smooth functions on $I$.

Although it is clear from Lemma \ref{Lminlem} that $L_-$ is not zero
as a formal pseudo-differential operator, nevertheless the coefficients
of $L_-$ evaluate to zero when viewed as functions on $I$.  That is, we have
the following result.

\begin{lem}
When $L$ is defined as in \eqref{genL}, with $\phi$ given by
\eqref{defphi}, then
\begin{equation*}
r(L_-)=0.
\end{equation*}
\label{Lmineq0}
\end{lem}

\begin{proof} From \eqref{Lmincoeff} we have that, as a formal
pseudo-differential operator,
\begin{equation*}
L_-\phi=\sum_{j=0}^{N-1}F_j\partial^j,
\end{equation*}
where each $F_j$ is a differential polynomial in
$W_0,\dots,W_{N-1}$.  Therefore
\begin{equation*}
r(L_-\phi)=\sum_{j=0}^{N-1}F_j(x)\partial^j,
\end{equation*}
where the $F_j(x)$ are smooth functions on $I$.

For all $i=1,\dots,N$, we see from \eqref{defphi} that
\begin{equation}
r(\phi) y_i = 0, \label{phiyieq0}
\end{equation}
and since $\partial^2 y_i=\alpha_i^2 y_i$, then
$r(\phi\partial^2)y_i=0$ also. Therefore \eqref{Dickey} implies
that
\begin{equation}
r(L_-\phi) y_i=\sum_{j=0}^{N-1}F_j(x)\partial_jy_i(x)=0 \
\text{for $i=1,\dots,N$}. \label{rlpyez}
\end{equation}
Since $D(x)=\det(\partial_jy_i(x))\ne 0$ for all $x \in I$, it
follows from \eqref{rlpyez}  that $F_j(x)=0$ for all $x \in I$ and
all $j=0,1,\dots,N-1$.

Now as a formal pseudo-differential operator, $\phi$ is
invertible, with inverse $\phi^{-1}$ of the form
\begin{equation*}
\phi^{-1}=\sum_{k=0}^{\infty}B_k \partial^{-N-k},
\end{equation*}
where the $B_k$ are differential polynomials in $W_0,\dots,
W_{N-1}$.  Hence
\begin{equation*}
L_-=(L_-\phi)\phi^{-1}=\sum_{j=0}^{N-1} F_j\partial^j
\sum_{k=0}^{\infty}B_k \partial^{-N-k}=\sum_{r=1}^\infty
\sum_{j=0}^{N-1}(F_j G_{j,r})\partial^{N-r},
\end{equation*}
where each $G_{j,r}$ is a finite linear combination of the $B_k$
and their formal derivatives $B_k'$, $B_k''$,\dots.  Since $F_j(x)=0$
for all $x \in I$, it follows that the coefficients
$\sum_{j=0}^{N-1}(F_j(x)G_{j,r}(x))$ of $r(L_-)$ are also identically zero
on $I$.
\end{proof}

The following consequence of Lemma \ref{Lmineq0} will be useful in Section
\ref{sec:2solconverse}.

\begin{cor}
If $I$ is any interval such that  $D(x) \ne 0$ for all $x \in I$, then the equation
\begin{equation}
W_{N-1}'-W_{N-1}^2+2W_{N-2}+\sum_{i=1}^N{\alpha_i^2} = 0
\label{wn1wn2}
\end{equation}
holds at all points of $I$.
\label{nicecor}
\end{cor}

\begin{proof}
From Lemma \ref{Lmineq0} we have $r(L_-)=0$, and hence each
coefficient in the sum in \eqref{Lmincoeff} is identically zero as a function
of $x$ on $I$.  In
particular,
\begin{equation*}
W_{N-1}''+2W_{N-2}'-2W_{N-1}'W_{N-1}=0
\end{equation*}
on $I$. Integrating gives
\begin{equation*}
W_{N-1}'+2W_{N-2}-W_{N-1}^2+C=0,
\end{equation*}
where $C$ is a constant.

To evaluate $C$, first assume that $\alpha_1$, \dots, $\alpha_N$
are positive numbers, and observe that since $y_i$ behaves as $x
\to \infty$ like $e^{\alpha_i x}$, we have that $\displaystyle
\lim_{x \to \infty} W_{N-1}=-d_1/d$, $\displaystyle\lim_{x \to
\infty} W_{N-1}' = 0$, and $\displaystyle\lim_{x \to \infty}
W_{N-2}=d_2/d$, where
\begin{equation*}
d= \left|
\begin{matrix}
1 & \dots & 1\\
\alpha_1 & \dots &\alpha_N\\
\alpha_1^2 & \dots &\alpha_N^2 \\
\dots & \dots & \dots  \\
\alpha_1^{N-1} & \dots & \alpha_N^{N-1}
\end{matrix}
\right|
\end{equation*}
is the Vandermonde matrix of the numbers
$\alpha_1$, \dots, $\alpha_N$, and
\begin{equation*}
d_1=\left|\begin{matrix}
1& \dots & 1\\
\alpha_1 & \dots &\alpha_N\\
\alpha_1^2 & \dots &\alpha_N^2\\
\dots & \dots & \dots  \\
\alpha_1^{N-2} & \dots & \alpha_N^{N-2}\\
\alpha_1^{N} & \dots & \alpha_N^{N}\\ \end{matrix}\right|, \quad
d_2= \left|\begin{matrix}
1 & \dots & 1\\
\alpha_1 & \dots &\alpha_N\\
\alpha_1^2 & \dots &\alpha_N^2\\
\dots & \dots & \dots  \\
\alpha_1^{N-3} & \dots & \alpha_N^{N-3}\\
\alpha_1^{N-1} & \dots & \alpha_N^{N-1}\\
\alpha_1^N & \dots & \alpha_N^N\\
\end{matrix}
\right|.
\end{equation*}
Therefore
\begin{equation}
C=\lim_{x \to \infty}\left(
W_{N-1}^2-2W_{N-2}-W_{N-1}'\right)=\left(\frac{d_1}{d}\right)^2-2\left(\frac{d_2}{d}\right).
\label{C}
\end{equation}
But it follows from a classic exercise on Vandermonde matrices
(see problem 10 on p.\ 99 of \cite{PS}, or \cite{MacS}) that
$$
d_1/d=\sum_{i=1}^n\alpha_i \quad\text{and} \quad d_2/d =
\sum_{1\le i < j \le n} \alpha_i\alpha_j.
$$
Substituting in \eqref{C}, we obtain $C=\sum_{i=1}^N{\alpha_i^2}$,
as desired.  The result for general complex values of $\alpha_1$,
\dots, $\alpha_N$ then follows by analytic continuation.
\end{proof}

\noindent
 \textit{Remark.} Since \eqref{wn1wn2} is a Ricatti
equation, the substitution $W_{N-1}=-D'/D$ converts it to the
following linear equation for $D$:
\begin{equation}
D''=\left(2W_{N-2}+\sum_{i=1}^N \alpha_i^2\right)D.
\label{lineqforD}
\end{equation}
Corollary \ref{nicecor} is therefore equivalent to the assertion
that \eqref{lineqforD} holds when $D$ is given by \eqref{Delta}
and \eqref{defyix}.

\begin{cor}
Suppose $D$ is given by \eqref{Delta}, $\phi$ by \eqref{defphi},
and $L$ by \eqref{genL}. Then
\begin{equation*}
r(L)=\partial^2+u,
\end{equation*}
where
\begin{equation}
u= 2\left(\frac{D'}{D}\right)'. \label{summarizeu}
\end{equation}
\end{cor}

\begin{proof} From the definition of the determinant and the product rule, one easily sees that
the derivative of $D$ is given by

\begin{equation}
D' = \left|
\begin{matrix}
y_1 & \dots &y_N  \\
y_1' & \dots &y_N' \\
\dots & \dots & \dots \\
y_1^{(N-2)} & \dots & y_N^{(N-2)}  \\
y_1^{(N)} & \dots & y_N^{(N)}  \\
\end{matrix}
\right|,
\end{equation}
which, as we see from \eqref{defWi}, implies  that $W_{N-1}=-D'/D$.  Since $r(L)=r(L_+)+r(L_-)$,
the desired
result therefore follows from Lemma \ref{Lposlem} and Lemma \ref{Lmineq0}.
\end{proof}

\begin{defn}  Let $y_j$ be given by \eqref{defyix}, and assume \eqref{paramcond} holds.   Then we define
\begin{equation}
\psi^{(N)}(x)=\psi^{(N)}(x;a_1,\dots,a_N;\alpha_1,\dots,\alpha_N)=2\left(\frac{D'(y_1,\dots,y_N)}{D(y_1,\dots,y_N)}\right)'
\label{itsmat}
  \end{equation}
for all $x$ such that $D(y_1,\dots,y_N) \ne 0$.
\label{defpsin}
\end{defn}

Introducing simple time dependencies into $\psi^{(N)}$
yields a function which satisfies all the equations in the KdV
hierarchy simultaneously.

\begin{thm} Let $N \in \mathbf N$, and let $a_j$, and $\alpha_j$, $j=1,\dots,N$ be complex numbers satisfying \eqref{paramcond}.  Fix $l \in \mathbf N$, and for
$1 \le j \le N$, define the function $\tilde a_j$ by
\begin{equation}
\tilde a_j(x,t_3,t_5,t_7,\dots,t_{2l+1})=a_j \exp\left(-2(\alpha_j^3 t_3 + \alpha_j^5 t_5 +
\dots + \alpha_j^{2l+1}t_{2l+1})\right). \label{defaitilde}
\end{equation}
Then let $u$ be defined as a function of $x$, $t_3$, $t_5$,
\dots, $t_{2l+1}$ by
\begin{equation}
u= \psi^{(N)}\left(x;\tilde a_1,\dots,\tilde
a_N;\alpha_1,\dots,\alpha_N\right). \label{summarizeu2}
\end{equation}
Then for all $k \in \{1,\dots,l\}$, and at all points in its
domain of definition, $u$ satisfies the partial differential
equation
\begin{equation}
\frac{\partial u}{\partial t_{2k+1}} =
\left[\left(L^{(2k+1)/2}\right)_+,L\right]=4(-1)^{k+1}R_{2k+3}'.
 \label{usolvesKdV}
\end{equation}
\label{solitonsoln}
\end{thm}

\noindent \textit{Remark.}  Using the fact that multiplication of
$y_i$ by the exponential of a linear function of $x$ does not
change the value of $\psi^{(N)}$, one sees easily that \eqref{summarizeu2} can also be written in the form
$$
u=2\left(\frac{D'(\tilde y_1,\dots,\tilde y_N)}{D(\tilde y_1,\dots,\tilde y_N)}\right)',
$$
where $\tilde y_j(x)=e^{w_j} + a_j e^{-{w_j}}$ and $w_j=\alpha_j x + \alpha_j^3 t_3 + \alpha_j^5 t_5 + \dots + \alpha_j^{2l+1}t_{2l+1}$.

\bigskip

We will not need to make use of Theorem \ref{solitonsoln} in the present paper, and so do not include a proof here.
But the reader may be interested to know that, using the tools defined in Section 2 above, a one-paragraph proof
can be given.  It may be found in \cite{Di}, where it appears as the proof of part (ii) of Proposition 1.6.5 in \cite{Di},
or in \cite{Di2} as the proof of Proposition 1.7.5.

We are concerned here, rather, with the fact the functions $\psi^{(N)}$ satisfy stationary equations of the form \eqref{stathier}:

\begin{thm}
Let $\psi^{(N)}$ be as in Definition \ref{defpsin}, and define
constants $s_0$, $s_1$, \dots, $s_N$ by
\begin{equation}
s_0 + s_1 x + s_2 x^2 + \dots +s_N x^N =
(x-\alpha_1^2)(x-\alpha_2^2)\cdots(x-\alpha_N^2); \label{defsi}
\end{equation}
in other words, $s_i$ is the $i^{\text th}$ elementary symmetric
function of $N$ variables, evaluated on $-\alpha_1^2$, \dots,
$-\alpha_N^2$. Then, on each interval of its domain of definition,
the function
\begin{equation*}
u(x) = \psi^{(N)}(x;a_1,\dots,a_N;\alpha_1,\dots,\alpha_N)
\end{equation*}
 satisfies the
ordinary differential equation in $x$ given by
\begin{equation}
s_0R_3 - s_1 R_5 + s_2 R_7 + \dots + (-1)^N s_N R_{2N+3} = C,
\label{almostsolitonequation}
\end{equation}
where $C$ is a constant.
 \label{forward}
\end{thm}

\begin{proof}
Define $\tilde a_i$ by \eqref{defaitilde} with $l=N$, and for $i
= 1,\dots, N$ extend $y_i$ to be a function of $x, t_3, t_5,
\dots, t_{2N+1}$ by replacing $a_i$ with $\tilde a_i$ in
\eqref{defyix}. That is, set
\begin{equation}
y_i = \exp(\alpha_i x) + a_i \exp\left[-(\alpha_i x + 2\alpha_i^3
t_3 + 2 \alpha_i^5 t_5 + \dots +2 \alpha_i^{2N+1}
t_{2N+1})\right]. \label{extyi}
\end{equation}
Also extend $u$ to be a function of $x, t_3, t_5, \dots, t_{2N+1}$
by \eqref{summarizeu2}.  With $\phi$ and $W_i$ defined in terms of
$y_i$ as before, we have as in \eqref{summarizeu} that
\begin{equation}
u=-2W_{N-1}'. \label{relateutow}
\end{equation}
   For $1\le k \le N$, let $\partial_{2k+1}$
denote differentiation with respect to $t_{2k+1}$.

For each $i$ from 1 to $N$, we apply to both sides of
\eqref{phiyieq0} the operator
\begin{equation*}
\tilde \partial = s_0 \partial + s_1 \partial_3 + s_2
\partial_5 + \dots + s_N \partial_{2N+1}.
\end{equation*}
There results the identity
\begin{equation}
 0=(\tilde\partial r(\phi)) y_i +r(\phi)(\tilde\partial y_i),
 \label{prodrule}
\end{equation}
where
\begin{equation*}
\tilde\partial r(\phi) = (\tilde\partial
W_{N-1})\partial^{N-1}+(\tilde\partial W_{N-2})\partial^{N-2}+
\dots + (\tilde\partial W_1)\partial + \tilde\partial W_0.
\end{equation*}
But for all $k = 1,\dots, N$ we have from \eqref{extyi} that
$$
\partial_{2k+1} y_i = \alpha_i^{2k+1} z_i,
$$
where
\begin{equation*}
z_i= -2a_i \exp\left[-(\alpha_i x + 2\alpha_i^3 t_3 + 2 \alpha_i^5
t_5 + \dots +2 \alpha_i^{2N+1} t_{2N+1})\right].
\end{equation*}
Therefore
\begin{equation} \tilde
\partial y_i = \left(s_0\alpha_i + s_1 \alpha_i^3+ s_2\alpha_i^5 +
\dots + s_N \alpha_i^{2N+1}\right)z_i. \label{tricky}
\end{equation}
It follows from \eqref{defsi} and \eqref{tricky} that
$\tilde\partial y_i=0$, and so, by \eqref{prodrule},
\begin{equation*}
(\tilde\partial r(\phi)) y_i = 0.
\end{equation*}

Now $\tilde\partial r(\phi)$  is a linear differential operator
in $x$ of order $N-1$ or less, so as in the proof of Lemma
\ref{Lmineq0}, the fact that it takes all the functions $y_1,
\dots, y_N$ to zero means that all its coefficients must be
identically zero. In particular, $\tilde\partial W_{N-1} = 0$, or,
in other words,
\begin{equation}
s_0 \partial W_{N-1} +s_1\partial_3 W_{N-1} +  \dots + s_N
\partial_{2N+1} W_{N-1}=0.
\label{eqnforw}
\end{equation}

On the other hand, \eqref{usolvesKdV} and \eqref{relateutow} tell
us that
\begin{equation}
-2\partial_{2k+1}W_{N-1}'=4(-1)^{k+1}R_{2k+3}'
\label{wpsolveskdvp}
\end{equation}
for $1 \le k \le l$.  Integrating \eqref{wpsolveskdvp} with
respect to $x$, we obtain
\begin{equation}
\partial_{2k+1}W_{N-1}=2(-1)^k R_{2k+3}+C_k,
 \label{wsolvesKdV}
\end{equation}
for $1 \le k \le l$, where $C_k$ is a constant of integration.
Moreover, from \eqref{Rrecur} we have $R_3 = -u/4$, and hence
\begin{equation}
\partial W_{N-1}=W_{N-1}'=2R_3. \label{wpR3}
\end{equation}
Substituting \eqref{wsolvesKdV} and \eqref{wpR3} into
\eqref{eqnforw} then gives \eqref{almostsolitonequation}.
\end{proof}

In general, $\psi^{(N)}$ will have singularities at points where
the denominator $D$ in \eqref{defphi} is equal to zero, but away
from these points, $\psi^{(N)}$ is a smooth, and in fact analytic,
function of its arguments.   Our next task is to determine
conditions on the parameters $\alpha_i$ and $a_i$ under which $D$
has no zeros on $\mathbf R$, or equivalently under which
$\psi^{(N)}$ is a smooth function on all of $\mathbf R$.

For this purpose it will be useful to represent $D$ as an
explicit sum of exponential functions.  For given $N
\in \mathbf N$, let $\{-1,1\}^N$ denote the set of functions
$\epsilon$ from $\{1,\dots,N\}$ to $\{-1,1\}$; thus $\{-1,1\}^N$
has cardinality $2^N$.  For $\epsilon \in \{-1,1\}^N$, we denote
the image of $j$ under $\epsilon$ by $\epsilon_j$, and define
$S(\epsilon)$ to be the set of all $j
\in \{1,\dots,N\}$ such that $\epsilon_j = -1$. Also, for any
ordered $N$-tuple $(r_1,\dots,r_N)$, let
\begin{equation*}
V(r_1,\dots,r_N)= \det\{r_i^{j-1}\}_{i,j=1,N}=\prod_{1 \le i < j
\le N} (r_j-r_i)
\end{equation*}
be the corresponding Vandermonde determinant.  Then expansion of the determinant in \eqref{Delta} yields
the formula
\begin{equation}
D=\sum_{\epsilon \in \{-1,1\}^N}\left[ \exp\left(\sum_{j=1}^N
\epsilon_j \alpha_j x\right)\left(\prod_{j\in
S(\epsilon)}a_j\right)
V(\epsilon_1\alpha_1,\dots,\epsilon_N\alpha_N)\right].
\label{firstforD}
\end{equation}

\begin{lem}
Let $y_j$ be given by \eqref{defyix}, and suppose that
\eqref{paramcond} holds.   Suppose in addition that for all $j \in
\{1,\dots,N\}$, $\Re \alpha_j > 0$.  Then $\psi^{(N)}$ and all of its
derivatives are defined for all sufficiently large $|x|$, and approach zero
exponentially fast as $|x| \to \infty$.
\label{psiatinf}
 \end{lem}

 \begin{proof}
Because $\Re \alpha_j > 0$ for all $j$, the dominant term in
\eqref{firstforD} is $Ve^{Sx}$, where
$V=V(\alpha_1,\dots,\alpha_N)$ and $S=\sum_{j=1}^N \alpha_j$:  all
other terms have exponents with smaller real parts.  In particular,
we have $D(x) \ne 0$ whenever $|x|$ is sufficiently large.  The
dominant terms in $D'(x)$ and $D''(x)$ are $VSe^{Sx}$ and
$VS^2e^{Sx}$, respectively; and so in the expression
$\psi^{(N)}(x)=2(DD''-(D')^2)/D^2$, the denominator has dominant
term $V^2e^{2Sx}$, while the coefficient of $e^{2Sx}$ in the
numerator is zero.  It follows easily that $\psi^{(N)}(x)$,
together with all its derivatives, tends to zero exponentially
fast as $x \to \infty$.  A similar argument applies as $x \to
-\infty$ (where the dominant term in $D(x)$ is $|a_1\cdots
a_NV|e^{|Sx|}$).
\end{proof}

\begin{lem}
Let $y_j$ be given by \eqref{defyix}, and suppose that
\eqref{paramcond} holds.   Suppose also that
for all $j \in \{1,\dots,N\}$, $\alpha_j$ and $a_j$ are real,
and
\begin{equation}
(-1)^{j-1}a_j > 0. \label{signai}
\end{equation}
 Then $\psi^{(N)}(x) \in H^1(\mathbf R)$.
 \label{condforDpos}
\end{lem}

\begin{proof}
Suppose that the $\alpha_j$ and $a_j$ are real and \eqref{signai}
holds. Then for $1 \le j < k \le N$ the factor
$\epsilon_k\alpha_k-\epsilon_j\alpha_j$ in
$V(\epsilon_1\alpha_1,\dots,\epsilon_N\alpha_N)$ has the same sign
as $\epsilon_k$.  For a given $k$, there are $k-1$ such factors in
$V(\epsilon_1\alpha_1,\dots,\epsilon_N\alpha_N)$, corresponding to
the values $1 \le j \le k-1$, so the sign of
$V(\epsilon_1\alpha_1,\dots,\epsilon_N\alpha_N)$ is $\prod_{k\in
S(\epsilon)}(-1)^{k-1}$.  It then follows from \eqref{signai} that
the coefficient of each exponential in \eqref{firstforD} is
positive.  Hence $D>0$ for all $x \in \mathbf R$, and it follows
that $\psi^{(N)}(x)$ is well-defined and smooth on all of $\mathbf
R$.  Then from Lemma \ref{psiatinf} it follows that $\psi^{(N)}
\in H^1$.
 \end{proof}

\begin{lem}  Let $y_j$ be given by \eqref{defyix}, and suppose that \eqref{paramcond} holds.
Suppose also that for each $j \in \{1,\dots,N\}$, either (i) $\alpha_j$ and $a_j$ are real, (ii)
$\alpha_j$ is purely imaginary and $|a_j|=1$, or (iii) there exists $k \in \{1,\dots,N\}$ such that
$\alpha_k = \alpha_j^\ast$ and $a_k = a_j^\ast$. (These conditions can be summarized by saying that
the numbers $\alpha_j^2$ and $(\log a_j)^2$ are either real and of the same sign, or occur in complex conjugate pairs.)
Then $\psi^{(N)}(x)$ is real-valued at all points where it is defined.
\label{psireal}
\end{lem}

\begin{proof} In case (i) we have $y_j^\ast = y_j$, in case (ii) we have $y_j^\ast = (1/a_j)y_j$, and in case (iii) we have
$y_j^\ast = y_k$.  It follows easily that the conjugate $D^\ast$
of $D=D(y_1,\dots,y_N)$ is equal to a constant times $D$ itself.
Therefore $(\psi^{(N)})^\ast = \psi^{(N)}$.
\end{proof}

\begin{defn} We say that $\psi^{(N)}(x;a_1,\dots,a_N;\alpha_1,\dots,\alpha_N)$ is an $N$-soliton profile
if $\psi^{(N)}(x)$ is real-valued for all $x \in \mathbf R$, and $\psi^{(N)}(x) \in H^1(\mathbf R)$.
The corresponding time-dependent functions given by
\eqref{summarizeu2} are called $N$-soliton solutions of the KdV
hierarchy. The numbers $\alpha_1$, \dots, $\alpha_N$ are called
the wavespeeds of the $N$-soliton solution. \label{defNsol}
\end{defn}

\noindent
 \textit{Remarks.}
(i) At least in the case when $N=2$, it can be shown that, given
that the conditions in \eqref{paramcond} hold for $a_j$ and
$\alpha_j$, the hypotheses on $a_j$ and $\alpha_j$ in
 Lemmas \ref{condforDpos} and \ref{psireal} are not only sufficient for $\psi^{(N)}$ to be an $N$-soliton profile according to the above definition, but also necessary.
 We conjecture that these conditions on $a_j$ and $\alpha_j$
 are also necessary in the case of general $N$, although we have not proved this yet. \\
  \indent (ii) If the conjecture in the preceding remark is true, it then follows from Lemma \ref{psiatinf} that $N$-soliton profiles,
together with all their derivatives, approach zero exponentially fast as $|x| \to \infty$.\\
\indent (iii) By transforming the index in the outermost sum of \eqref{firstforD}
from $\epsilon$ to $\mu$, where $\mu_j=\frac12(\epsilon_j+1)$ for
$j \in \{1,\dots,N\}$, one can rewrite $D$ in the form $D =
e^{px+q} D_1$, where $p$ and $q$ are constants and
\begin{equation*}
D_1=\sum_{\mu \in \{0,1\}^N} \exp\left(\sum_{i=1}^N 2\mu_i
\alpha_i(x+\zeta_i)+\sum_{1 \le i < j \le n} \mu_i\mu_j
A_{ij}\right),
\end{equation*}
where $\zeta_i$ and $A_{ij}$ are real constants. Explicitly, one
has
\begin{equation*}
\begin{aligned}
e^{px+q}&=\exp\left(-\Sigma_{i=1}^N \alpha_i
x\right)V(\alpha_1,\dots,\alpha_N) \Pi_{j=1}^N |a_j|,\\
\zeta_i
&=\frac{1}{2\alpha_i}\log\left|\frac{V(\alpha_1,\dots,\alpha_{i-1},-\alpha_i,\alpha_{i+1},\dots,\alpha_N)}{a_i
V(\alpha_1,\dots,\alpha_N)}\right|,\quad (i=1,\dots,N)\\
A_{ij}&=2\log\left|\frac{\alpha_j-\alpha_i}{\alpha_j+\alpha_i}\right|, \quad(i,j =1,\dots, N).
\end{aligned}
\end{equation*}
(Here use has been made of the assumption \eqref{signai}.) Writing
\begin{equation*}
\psi^{(N)}=2(D_1'/D_1)',
\end{equation*}
one obtains the formula for $N$-soliton profile found in \cite{MS} or
on page 55 of \cite{H2}.\\
\indent (iv) If $\psi^{(N)}$ is an $N$-soliton solution, then the constant $C$ in equation \eqref{almostsolitonequation}
is equal to zero; that is,
\begin{equation}
s_0R_3 - s_1 R_5 + s_2 R_7 + \dots + (-1)^N s_N R_{2N+3} = 0
\label{solitonequation}
\end{equation}
on $\mathbf R$.  This is seen by taking
the limit of \eqref{almostsolitonequation} as $x \to \infty$, and observing that, for each
$k \ge 1$,  $R_{2k+1}$   is a differential polynomial in $u=\psi^{(N)}$ and its derivatives, with no constant term.

\section{The stationary equation for $N=1$}
\label{sec:1solconverse}

To set the stage for the analysis of \eqref{stathier} in the case
$N=2$, we now discuss the case $N=1$.  The result we prove in this
section, Theorem \ref{1solthm}, is a standard exercise in
elementary integration, but writing out the proof in detail will
serve to introduce the notation we use for the more complicated
computations of the next section.

From \eqref{R1357}, we have that in the case when $N=1$, \eqref{stathier} is given by
 \begin{equation}
\frac{d_1}{2} -\frac{d_3}{4}(u) + \frac{d_5}{16} (u''+3u^2)=0.
\label{stathier1}
\end{equation}

Suppose $d_1$, $d_3$, and $d_5$ are given real numbers, and
suppose $u \in L^2$ is a real-valued solution of
the ordinary differential equation \eqref{stathier1} (in the sense
of distributions).  Then from \eqref{stathier1} we see that $d_5$
must be nonzero, and by dividing by $d_5$ if necessary, we can
assume that $d_5 = 1$. Also, multiplying both sides of
\eqref{stathier1} by a test function $\phi_\tau(x)=\phi(x-\tau)$,
where $\int_{\mathbf R}\phi(x)\ dx = 1$, and letting $\tau \to
\infty$, we arrive at the conclusion that
\begin{equation}
\lim_{\tau \to \infty} \int_{\mathbf R} u \phi_\tau'' = \lim_{\tau
\to \infty}\int_{\mathbf R} u'' \phi_\tau = \lim_{\tau \to
\infty}\int_{\mathbf R}(-3u^2 +4d_3 u -8d_1) \phi_\tau  = - 8d_1,
\label{d1eq0}
\end{equation}
from which it follows that $d_1 = 0$.  Letting $C=d_3$, we can
then write \eqref{stathier1} as
\begin{equation}
C\left(\frac{-u}{4}\right)+\frac{1}{16}(u''+3u^2)= 0.
\label{1soleqn}
\end{equation}

\begin{lem}  Suppose $u \in L^2$ is a solution of
\eqref{1soleqn} in the sense of distributions. Then $u$ must be in
$H^s$ for all $s \ge 0$, and $u$ is analytic on $\mathbf R$.
\label{regularity1}
\end{lem}

\begin{proof}
 Equation \eqref{1soleqn} can be rewritten as
\begin{equation}
u-u''=au+bu^2, \label{boot}
\end{equation}
where $a$ and $b$ are constants.  Let $\cal F$ denote the Fourier
transform, defined for $f \in L^1$ by ${\cal F}f(k)=\intR
f(x)e^{ikx}\ dx$, and extended to $L^2$ in the usual way. Letting
$f=au$ and $g=bu^2$, taking the Fourier transform of \eqref{boot},
and dividing both sides by $(1+k^2)^{1/2}$, we obtain
$$
(1+k^2)^{1/2}{\cal F}u = \frac{1}{(1+k^2)^{1/2}}({\cal F}f + {\cal
F}g).
$$
Since $f$ is in $L^2$, then ${\cal F}f$ and ${\cal
F}f/(1+k^2)^{1/2}$ are in $L^2$; and since $g$ is in $L^1$, then
${\cal F}g$ is bounded and continuous, and ${\cal
F}g/(1+k^2)^{1/2}$ is in $L^2$. Therefore $(1+k^2)^{1/2}{\cal F}u
\in L^2$, so $u \in H^1$.  But it then follows that $u^2 \in L^2$,
whence both $f$ and $g$ are in $L^2$, so \eqref{boot} gives $u''
\in L^2$ and $u \in H^2$. Taking derivatives of \eqref{boot}
successively now easily gives that all higher-order derivatives of
$u$ are in $L^2$, so that $u \in H^s$ for all $s \ge 0$.

As a particular consequence, we have that $u$ is a classical
solution of \eqref{boot} on $\mathbf R$.  Therefore the fact that
$u$ is analytic on $\mathbf R$ follows from the fundamental
theorems of ordinary differential equations, given that the
right-hand side of \eqref{boot} is an analytic function of $u$
(see, e.g., section 1.8 of \cite{CL}).
\end{proof}

\begin{thm} Suppose $C \in \mathbf R$, and suppose $u \in L^2$ is a real-valued solution of
\eqref{1soleqn}, in the sense of distributions.  Suppose also that $u$ is not identically
zero.  Then $C>0$, and there exists $K \in \mathbf
R$ such that
\begin{equation}
u=\frac{2C}{\cosh^2(\sqrt C x + K)}. \label{cosh2}
\end{equation}
\label{1solthm}
\end{thm}

\noindent
 \textit{Remark.} We can also write \eqref{cosh2} in the
form
\begin{equation*}
u=\psi^{(1)}(x;a;\sqrt{C}),
\end{equation*}
where $a=e^{-2K}$.

\bigskip

\begin{proof}
Suppose $u$ is an $L^2$ solution of \eqref{1soleqn} on $\mathbf R$,
and is not identically 0.   By Lemma
\ref{regularity1}, $u$ is analytic on $\mathbf R$, and is in $H^s$ for every $s \ge 0$.
Hence, in particular, $u$ and all its derivatives tend to zero as $|x| \to \infty$.

Multiplying \eqref{1soleqn} by $u'$ and integrating gives
\begin{equation}
(u')^2 = 4Cu^2 - 2u^3, \label{uprimesq}
\end{equation}
where we have used the fact that $u \to 0$ and $u' \to 0$ as $x
\to \infty$ to evaluate the constant of integration as zero.
Letting $\zeta = -C + u/2$, we can rewrite \eqref{uprimesq} as
\begin{equation}
(\zeta')^2= -4\zeta (\zeta+C)^2.
 \label{zetaprimesq}
\end{equation}

 Since $u$ is analytic on $\mathbf R$, then so is $\zeta$.  We know that
 $\zeta$ cannot be identically equal to $-C$ on $\mathbf R$, because $u$ is not
 identically zero.  Also, if $C \ne 0$, then $\zeta$ cannot be identically equal
 to $0$ on $\mathbf R$, because this would contradict the fact that $u \to 0$ as $x \to \infty$.
 Therefore the set $\left\{x \in \mathbf
R: \zeta(x)=0\ \text{or}\ \zeta=-C\right\}$ must consist of isolated points (or be empty).  Hence
there exists an open interval $I$ in $\mathbf R$ such that for
all $x \in I$,  $\zeta(x) \ne 0$ and $\zeta(x) +C \ne 0$.  Also,
from \eqref{zetaprimesq} it follows that $\zeta(x) < 0$ for $x
\in I$.

Define $\Omega$ to be the domain
in the complex plane given by
\begin{equation}
\Omega = \mathbf C -
\{z: \Re z \ge 0 \ \text{and $\Im z =0$}\}.
\label{defOmega}
\end{equation}
  Henceforth, for $z \in \Omega$
we will denote by $\sqrt z$ the branch of the
square root function given by $\sqrt z =
\sqrt r e^{i\theta/2}$ when $z=re^{i\theta}$ with $r >0$ and $0 <
\theta < 2\pi$.  Thus $\sqrt z$ is an analytic function on $\Omega$, and since $\zeta(x)$
takes values in $\Omega$, then $\sqrt{\zeta(x)}$ is an analytic function of $x$ on $I$.

 From \eqref{zetaprimesq} we have that there exists a
function $\theta:I \to \left\{-1,1\right\}$ such that
\begin{equation}
\zeta'=2i\theta(x)\sqrt{\zeta}(\zeta+C)
\label{zetaprime}
\end{equation}
for all $x \in I$.   Since
$\zeta(x)+C \ne 0$ for all $x \in I$, it then follows from
\eqref{zetaprime} that $\theta$ is analytic on $I$, and, since
$\theta$ takes values in $\left\{-1,1\right\}$, $\theta$ must
therefore be constant on $I$.

Now define $v =-i\theta\sqrt{\zeta}$ on $I$, noting for future
reference that, since $\zeta < 0$ on $I$, then $v$ is real-valued.
Let $\alpha = \sqrt C$. We then have from \eqref{zetaprime} that
\begin{equation}
 \frac{v'}{C-v^2}= \frac{v'}{\alpha^2-v^2}= 1.
 \label{vprime}
\end{equation}

To integrate \eqref{vprime}, we first fix $x_0 \in I$, let $V=v(x_0)$,
and define
\begin{equation*}
L_{\alpha,V}(z)=\int_V^z{\frac{dw}{\alpha^2-w^2}}.
\end{equation*}
Since $V \ne \pm \alpha$, this defines $L_{\alpha,V}$ as a single-valued,
analytic function of $z$ in some neighborhood of $V$.  By shrinking $I$
if necessary, we may assume that $L_{\alpha,V}(v(x))$ is defined for all $x \in I$,
and so \eqref{vprime} may be integrated to give
\begin{equation}
L_{\alpha,V}(v(x))=x-x_0
\label{intvprime}
\end{equation}
for $x \in I$.

Our next goal will be to solve \eqref{intvprime} for $v(x)$. Once
this has been done, we can recover $u$ from the formula
\begin{equation}
u=2(\zeta+C)=2(\alpha^2-v^2). \label{recovuneq1}
\end{equation}

Consider first the case when $\alpha \ne 0$ (and hence $C \ne 0$).  In this case, by
choosing an appropriate branch of the complex logarithm function,
we could express $L_{\alpha,V}(z)$ as
\begin{equation}
L_{\alpha,V}(z)=\frac{1}{2\alpha}\left(\log\left(\frac{\alpha+z}{\alpha+V}\right)-\log\left(\frac{\alpha-z}{\alpha-V}\right)\right).
\label{defG0}
\end{equation}
However, this will not be necessary, since we really only need to
use the fact that
\begin{equation}
 \exp(2 \alpha L_{\alpha,V}(z))=\left(\frac{\alpha + z}{\alpha
- z}\right)\left(\frac{\alpha - V}{\alpha + V}\right) \label{eL}
\end{equation}
 for all $z$
in some neighborhood of $V$.  To see that \eqref{eL} is true,
define $f_1(z)$ to be the function on the left side of \eqref{eL},
and $f_2(z)$ to be the function on the right side.    Then both
$f_1$ and $f_2$ satisfy the differential equation $df/dz= (2 \alpha/(\alpha^2-z^2)) f(z)$ in some neighborhood of $V$, and both take the value 1 at
$z=V$. Since a solution $f$ of the differential equation with a prescribed
value at $V$ is unique on any neighborhood of $V$ where it is defined, $f_1$ must equal $f_2$ on some
neighborhood of $V$.

Now multiplying both sides of \eqref{intvprime} by $2\alpha$,
taking exponentials, and using \eqref{eL}, one obtains
\begin{equation}
\frac{\alpha +v}{\alpha - v} = \left(\frac{\alpha + V}{\alpha -V}\right)e^{2\alpha(x-x_0)}=e^{2A}, \label{eqforv}
\end{equation}
where
\begin{equation}
A = \alpha (x-x_0)+M
\label{defA}
\end{equation}
and $M$ is any number such that
\begin{equation}
e^{2M}=\frac{\alpha+V}{\alpha-V}.
\label{defK}
\end{equation}
  Solving \eqref{eqforv} for $v$, we
find that
\begin{equation*}
v=\frac{y'}{y}
\end{equation*}
where
\begin{equation}
y = \sinh A.
\label{formforyneq1}
\end{equation}
Substituting into \eqref{recovuneq1}, we find that
\begin{equation}
u=2\left(\frac{\alpha^2y^2 - (y')^2}{y^2}\right)=2(y'/y)'.
\label{solnforuneq1}
\end{equation}

Since $u$ is analytic on $\mathbf R$, then the function on the
right side of \eqref{solnforuneq1} is extendable to an analytic
function on $\mathbf R$.  This implies that
$y$ cannot have any zeroes on $\mathbf R$. We now have to
determine the values of $C$ for which this is possible.  We
consider separately the subcases in which $C>0$ and $C<0$.

If $C>0$, then $\alpha=\sqrt C$ is real, and from \eqref{defA} and \eqref{defK}
we see that we can take
\begin{equation}
A= \alpha x + K + i \sigma \pi/2,
\label{e2areal1}
\end{equation}
where $K$ is real and either $\sigma =0$ or $\sigma = 1$, according
to whether $(V+\alpha)/(V-\alpha)$ is positive or negative.  If
$\sigma =0$, then
\begin{equation*}
y=\sinh (\alpha x +K),
\end{equation*}
which equals zero for some $x \in \mathbf R$, so $u$ has a
singularity at this $x$.  On the other hand, if $\sigma =1$, then
\eqref{formforyneq1} gives
\begin{equation*}
y = i\cosh(\alpha x + K),
\end{equation*}
which does not vanish at any point of $\mathbf R$. In this case
the function $u$ given by \eqref{solnforuneq1} is nonsingular, and
in fact we recover the solution given by \eqref{cosh2}.

If, on the other hand, $C<0$, then $\alpha=i\sqrt {|C|}$ is purely
imaginary, so
\begin{equation*}
\left|\frac{\alpha+V}{\alpha -V}\right|=1.
\end{equation*}
It then follows from \eqref{defA} and \eqref{defK} that $A$ is purely imaginary,
and we can write
\begin{equation}
A= i(\sqrt {|C|} x +K),
\label{pureimconj2}
\end{equation}
where $K$ is real.  Then \eqref{formforyneq1} gives
\begin{equation*}
y=i \sin(\sqrt{ |C|} x + K),
\end{equation*}
contradicting the fact that $y$ cannot have any zeroes on $\mathbf R$.  We conclude that $C$ cannot be
negative.

It remains to show that $\alpha$ and $C$ cannot equal zero. For if they were, then integrating \eqref{vprime} would give
$v=(x+K)^{-1}$, where $K$ is a constant, and so by \eqref{recovuneq1},
\begin{equation*}
u=\frac{-2}{(x+K)^2}
\end{equation*}
for all $x \in I$.  But this contradicts the fact that $u$ is analytic on all of $\mathbf R$.

We have now shown that if \eqref{1soleqn} has a solution in $L^2$
that is not identically zero, then $C$ must be positive, and in
that case the only solutions are those given by \eqref{cosh2}. So
the proof is complete.
\end{proof}

\section{The stationary equation for $N=2$} \label{sec:2solconverse}

According to Theorem \ref{forward} and Definition \ref{defNsol},
every $N$-soliton solution of the KdV hierarchy has profiles which
are solutions of the stationary equation \eqref{stathier}, or more
specifically of \eqref{solitonequation}. In this section, for the
case $N=2$, we prove a converse to this result: every $H^2$
solution of \eqref{stathier} with $N=2$ must be either a
$1$-soliton profile or a $2$-soliton profile.

Taking $N=2$ in
\eqref{stathier}, we obtain from \eqref{R1357} the equation
\begin{equation}
\frac{d_1}{2}-\frac{d_3}{4}(u)+\frac{d_5}{16}(u''+3u^2)-\frac{d_7}{64}
(u''''+5u'^2+10uu''+10u^3)=0. \label{stathier2}
\end{equation}
We may assume in what follows that $d_7 \ne 0$, for otherwise we
are back in the case $N=1$, which has already been handled in
section \ref{sec:1solconverse}. Dividing by $d_7$ if necessary, we
can therefore take $d_7=1$ without losing generality.   We may
also henceforth assume that $d_1=0$, since a computation similar
to that given in \eqref{d1eq0} shows that this must be the case if
\eqref{stathier2} has a solution $u$ in $H^2$.

\begin{lem}  Suppose $u \in H^2$ is a solution of
equation \eqref{stathier2} in the sense of distributions. Then $u$ must be
in $H^s$ for all $s \ge 0$, and $u$ is analytic on $\mathbf R$. In
particular, we have
\begin{equation}
\lim_{x \to \infty} u(x) = \lim_{x \to \infty} u'(x)=\lim_{x \to
\infty} u''(x)=\lim_{x \to \infty} u'''(x)=0.
 \label{zeroatinfty}
\end{equation}
\label{regularity2}
\end{lem}

\begin{proof}
Taking $d_1 = 0$ in \eqref{stathier2} and solving for $u''''$, we
obtain
\begin{equation}
u''''=au+bu^2+cu^3+d(u')^2+eu''+fuu'', \label{solveforu4}
\end{equation}
where $a, b, c, d, e, f$ are constants.  Since $u \in H^2$, then
all the terms on the right-hand side of \eqref{solveforu4} are in
$L^2$, so $u'''' \in L^2$ as well. Hence $u \in H^4$, and this
already yields \eqref{zeroatinfty}.  It also implies that $u$ is a
classical solution of \eqref{stathier2}, so by fundamental
theorems of ordinary differential equations, $u$ is analytic.
Finally, taking derivatives of \eqref{solveforu4} successively and
applying an inductive argument yields that $u \in H^s$ for all $s
\ge 0$.
\end{proof}

\begin{thm}
Suppose $d_1=0$, $d_7=1$, and $d_3$ and $d_5$ are arbitrary real numbers, and suppose $u \in H^2$ is a nontrivial (i.e., not identically zero)
distribution solution of
equation \eqref{stathier2}.

   Then either

(i) $u$ is a 1-soliton profile given by
\begin{equation}
u= \psi^{(1)}\left(x;a;\sqrt{C}\right), \label{casei}
\end{equation}
where $C$ is a positive root of the quadratic equation
\begin{equation}
z^2-d_5z+d_3=0. \label{c1c2}
\end{equation}
and $a$ is a real number such that $a > 0$; or

(ii) $u$ is a 2-soliton profile given by
\begin{equation}
u=\psi^{(2)}\left(x;a_1,a_2;\sqrt{C_1},\sqrt{C_2}\right),
\label{desiredu}
\end{equation}
where $C_1$ and $C_2$ are roots of equation \eqref{c1c2} with $0 <
C_1 < C_2$, and $a_1$ and $a_2$ are real numbers such that $a_2 <
0 <a_1$.

\label{2solthm}
\end{thm}

\begin{proof} Suppose $u$ is a
nontrivial distribution solution of \eqref{stathier2} with $d_7=1$
and $d_1=0$. By Lemma \ref{regularity2} we may assume that $u$ is
analytic on $\mathbf R$ and satisfies \eqref{zeroatinfty}.

 Following Chapter 12 of \cite{Di} we define, for $x \in \mathbf R$ and $\zeta \in \mathbf C$,
\begin{equation}
\hat R(x,\zeta)=\hat R_0 + \hat R_1 \zeta + \hat R_2 \zeta^2,
\label{defRhat}
\end{equation}
where
\begin{equation}
\begin{aligned}
\hat R_0 &= d_3 R_1 + d_5 R_3 + d_7R_5 = \frac{d_3}{2}-\frac{d_5}{4}\ u+\frac{1}{16}(u''+3u^2)\\
\hat R_1 &= d_5R_1 + d_7 R_3 =\frac{d_5}{2}-\frac14\ u\\
\hat R_2 &= d_7R_1 = \frac12\ .
\end{aligned}
\label{defhatRi}
\end{equation}
We claim that
\begin{equation}
\hat R'''+4u\hat R'+2u'\hat R+4\zeta \hat R'=0.
 \label{eqnforRhat}
 \end{equation}
 Indeed, substituting \eqref{defRhat} into \eqref{eqnforRhat} and using Lemma \ref{recrel}, we find that the left side
 of \eqref{eqnforRhat} is equal to
  \begin{equation}
\begin{aligned}
  -4(d_3R_3'&+d_5R_5'+d_7R_7')+
  \\
&+\zeta(d_5Q_1+d_7Q_3+4d_3R_1')+\zeta^2(d_7Q_1+4d_5R_1')+\zeta^3(4d_7R_1'),
\end{aligned}
 \label{uglyeqforRs}
 \end{equation}
where $Q_1=R_1'''+4uR_1'+2u'R_1+4R_3'$ and $Q_3=R_3'''+4uR_3'+2u'R_3+4R_5'$.
 But since $u$ is a solution of \eqref{stathier2} and $d_1=0$, we have that $d_3R_3+d_5R_5+d_7R_7=0$, so the first term
in \eqref{uglyeqforRs} vanishes, and $Q_1$ and $Q_3$ are zero by virtue of Lemma \ref{recrel}.  Since $R_1'=0$, this proves \eqref{eqnforRhat}.

Multiplying \eqref{eqnforRhat} by $\hat R$ and integrating with
respect to $x$ gives
\begin{equation}
2\hat R'' \hat R -\hat R'^2 + 4(u+\zeta)\hat R^2 = P(\zeta),
\label{inteqnforRhat}
\end{equation}
where $P(\zeta)$ is a polynomial in $\zeta$ with coefficients that are independent of $x$.
From \eqref{R1357}, \eqref{zeroatinfty}, and \eqref{defhatRi} we
see that
\begin{equation}
\begin{aligned}
\lim_{x \to \infty} \hat R_0 &= d_3/2 \\
\lim_{x \to \infty} \hat R_1 &= d_5/2 \\
\lim_{x \to \infty} \hat R_2 &= d_7/2 = 1/2. \\
\end{aligned}
\label{hatRatinf}
\end{equation}
Also,
\begin{equation*}
\lim_{x \to \infty} \hat R_i'(x)=\lim_{x \to \infty}\hat
R_i''(x)=0 \quad \text{for $i=0,1,2$}.
\end{equation*}
Therefore, taking the limit of \eqref{inteqnforRhat} as $x \to
\infty$, we get that
\begin{equation}
P(\zeta)=\zeta(d_3+d_5\zeta + \zeta^2)^2. \label{P}
\end{equation}
Combining \eqref{inteqnforRhat} and \eqref{P} gives
\begin{equation}
2\hat R'' \hat R -\hat R'^2 + 4(u+\zeta)\hat R^2
=\zeta(d_3+d_5\zeta + \zeta^2)^2. \label{basiceqn}
\end{equation}
Let $C_1$ and $C_2$ denote the (possibly repeated) roots of
equation \eqref{c1c2}.  As roots of a polynomial with real
coefficients, $C_1$ and $C_2$ are either both real numbers or are
complex conjugates of each other, and we may assume they are
ordered so that $\Re C_1 \le \Re C_2$. Then
\begin{equation}
d_3 = C_1C_2\quad\text{and}\quad  d_5 = C_1+C_2,
 \label{mustdi}
\end{equation}
and so, by \eqref{basiceqn},
\begin{equation}
2\hat R'' \hat R -\hat R'^2 + 4(u+\zeta)\hat R^2
=\zeta(\zeta+C_1)^2(\zeta+C_2)^2. \label{PCi}
\end{equation}

Let us now view the function $\hat R(x,\zeta)$ as a polynomial in the complex variable
$\zeta$ with coefficients which are analytic functions of $x$.  Our next goal is to study the roots
of this polynomial.

First, observe that since $\hat R_0$ is, like $u$, analytic on
$\mathbf R$, then $\hat R_0$ is either identically zero on
$\mathbf R$ or has only isolated zeros. But if $\hat R_0$ is
identically zero on $\mathbf R$, then by \eqref{hatRatinf} we must
have $d_3=0$.   The equation $\hat R_0 =0$ in \eqref{defhatRi} is
then seen to take the form of \eqref{1soleqn}, with $C$ replaced
by $d_5$, and so it follows from Theorem \ref{1solthm} that $d_5
>0$ and $u$ is given by \eqref{casei}.  Notice also that
since $d_3=0$, $d_5$ is a positive root of \eqref{c1c2}.  Thus
statement (i) of the Theorem holds in this case.  Therefore, we
can, without loss of generality, assume that $\hat R_0$ has only
isolated zeros, and hence there exists an open interval $I
\subseteq \mathbf R$ such that $\hat R_0(x) \ne 0$ for all $x \in
I$. It then follows that for all $x \in I$, $\zeta = 0$ is not a
root of $\hat R(x,\zeta)$.

We claim that there exists at least one $x_0 \in I$ such that the
polynomial $\hat R(x_0,\zeta)$ has distinct roots $\zeta_1$ and
$\zeta_2$. For if this is not the case, then there exists a
function $\zeta_1(x)$ such that for all $x \in I$,
\begin{equation}
\hat R(x,\zeta)=\frac12 (\zeta - \zeta_1(x))^2.
 \label{doubleroot}
\end{equation}
From \eqref{doubleroot} and \eqref{defRhat} we have that
$\zeta_1(x)^2=2\hat R_0(x)$, and since $\hat R_0(x)$ is nonzero
for all $x \in I$ it follows that $\zeta_1(x)$ is analytic, and
hence differentiable, as a function of $x$.  Thus we can
differentiate \eqref{doubleroot} with respect to $x$ to obtain
\begin{equation*}
\hat R'(x,\zeta)=-(\zeta-\zeta_1(x))\zeta_1'(x)
\end{equation*}
for $x \in I$. But then substituting $\zeta=\zeta_1(x)$ into
\eqref{PCi} gives
\begin{equation*}
0 = \zeta_1(x)(\zeta_1(x)+C_1)^2(\zeta_1(x)+C_2)^2.
\end{equation*}
Since $\zeta_1(x) \ne 0$ for $x \in I$, it follows that for all $x
\in I$, either $\zeta_1(x)=-C_1$ or $\zeta_1(x)=-C_2$.  Since
$\zeta_1(x)$ is analytic on $I$, the set of points where
$\zeta_1(x)$ takes a given value must either be a discrete subset
of $I$, or consist of all of $I$.  Since the union of two discrete
subsets of $I$ cannot equal all of $I$, it must be that either
$\zeta_1(x) \equiv -C_1$ on $I$ or $\zeta_1(x) \equiv -C_2$ on
$I$. Hence \eqref{doubleroot} gives, for either $j=1$ or $j=2$,
\begin{equation*}
\hat R(x,\zeta)= \frac12 (\zeta^2 + 2C_j \zeta + C_j^2),
\end{equation*}
and so, by \eqref{defRhat},
\begin{equation*}
\hat R_1 = \frac{d_5}{2} - \frac14 u = C_j
\end{equation*}
holds for all $x \in I$.  But this implies that $u$ is constant on
$I$, and since $u$ is analytic on $\mathbf R$, then $u$ must be
constant on $\mathbf R$.  Then \eqref{zeroatinfty} gives that $u$
is identically zero, contrary to our assumption that $u$ is
nontrivial. Thus the claim has been proved.

It now follows from standard perturbation theory \cite{K} that, by
shrinking $I$ if necessary to a smaller neighborhood, we can
assume that there exist analytic, nonzero functions $\zeta_1$ and
$\zeta_2$ on $I$ such that $\zeta_1(x) \ne \zeta_2(x)$ and $\hat
R(x,\zeta_1(x))=\hat R(x,\zeta_2(x))=0$ for all $x \in I$. We
therefore have
\begin{equation}
\hat R(x,\zeta)=\frac12(\zeta-\zeta_1(x))(\zeta-\zeta_2(x)),
\label{rootsofR}
\end{equation}
for all $x \in I$. Also, since $\zeta_1$ and $\zeta_2$ are roots of a real polynomial, we have that either $\zeta_1$ and $\zeta_2$ are both
real on $I$ or $\zeta_1^\ast = \zeta_2$ on $I$.

Our goal in what follows is to obtain a second-order system of
differential equations for $\zeta_1$ and $\zeta_2$, which can
then be integrated explicitly to find $\zeta_1$ and $\zeta_2$.
Once this is accomplished, it is easy to recover $u$, since
\eqref{defRhat} and \eqref{rootsofR} imply that  $\zeta_1 +\zeta_2 = -2\hat R_1$,
and hence
\begin{equation}
u=2(\zeta_1 + \zeta_2 + d_5).  \label{uzetai}
\end{equation}

For $i=1,2$, we have $\hat R(x,\zeta_i(x)) =0$, and hence it
follows from \eqref{PCi} that
\begin{equation}
\hat R'(x,\zeta_i(x))^2 = -\zeta_i(\zeta_i+C_1)^2(\zeta_i+C_2)^2 .
\label{rprimeatzero}
\end{equation}
Differentiating \eqref{rootsofR} with respect to $x$, we obtain
$$
\hat
R'(x,\zeta)=-\frac12\left[(\zeta-\zeta_1)\zeta_2'(x)+(\zeta-\zeta_2)\zeta_1'(x)\right].
$$
Therefore
\begin{equation*}
\begin{aligned}
\hat R'(x,\zeta_1(x)) &= -\frac12
\left(\zeta_1(x)-\zeta_2(x)\right)\zeta_1'(x)\\
\hat R'(x,\zeta_2(x)) &= -\frac12
\left(\zeta_2(x)-\zeta_1(x)\right)\zeta_2'(x).
\end{aligned}
\end{equation*}
From \eqref{rprimeatzero}  we then have that
\begin{equation}
\begin{aligned}
(\zeta_1-\zeta_2)^2 (\zeta_1')^2 &= -4
\zeta_1(\zeta_1+C_1)^2(\zeta_1+C_2)^2\\
(\zeta_1-\zeta_2)^2 (\zeta_2')^2 &= -4 \zeta_2(\zeta_2+C_1)^2
(\zeta_2+C_2)^2.
\end{aligned}
\label{zetasys}
\end{equation}

Note that, since $C_1$ and $C_2$ are either both real numbers, or
are complex conjugates of one another, it follows from
\eqref{zetasys} that if $\zeta_1$ and $\zeta_2$ are real, they
must necessarily take negative values at all points of $I$.  Since
 $\zeta_1$ and $\zeta_2$ are nonzero functions on $I$, it follows that both
$\zeta_1$ and $\zeta_2$ map $I$ into the domain $\Omega$ in the
complex plane defined in \eqref{defOmega}.  Also, since
$\zeta_1-\zeta_2$ has no zeros in $I$, we have from
\eqref{zetasys} that
 \begin{equation}
 \begin{aligned}
\zeta_1'&=\frac{2i\theta_1(x)\sqrt{\zeta_1}(\zeta_1 +C_1)(\zeta_1+C_2)}{\zeta_1-\zeta_2}\\
\ &\ \\
 \zeta_2'&=\frac{2i\theta_2(x)\sqrt{\zeta_2}(\zeta_2+C_1)(\zeta_2+C_2)}{\zeta_1-\zeta_2}\
,
 \end{aligned}
\label{sys2}
\end{equation}
where $\theta_i(x) \in \{-1,1\}$ for $i=1,2$. Here, as throughout
the paper, we use $\sqrt{z}$ to denote the analytic branch of the
square root function on $\Omega$ defined after \eqref{defOmega}.

The change of variables from $u$ to $(\zeta_1,\zeta_2)$, which
reduces the stationary equation \eqref{stathier2} to the separable
system \eqref{sys2}, is due originally to Dubrovin in \cite{Du}
(see also \cite{DN, IM, N}, and chapter 12 of \cite{Di}).  These
authors use the same change of variables (or, more precisely, its
generalization to the case of general $N$) to, among other things,
determine the time evolution of finite-gap solutions of the
Korteweg-de Vries hierarchy.

Again using the analyticity of $\zeta_1$ and $\zeta_2$, and taking
$I$ smaller if necessary, we can reduce consideration to the
following two cases: either there exist $i, j \in \{1,2\}$ such
that
\begin{equation}
\zeta_i(x) + C_j = 0 \quad\text{for all $x \in I$},
\label{zetaipcjezero}
\end{equation}
or, for all $i,j \in \{1,2\}$,
\begin{equation}
\zeta_i(x) + C_j \ne 0 \quad\text{for all $x \in I$}.
\label{zetipcjnezero}
\end{equation}

Suppose \eqref{zetaipcjezero} holds, with for example $i=1$ and
$j=1$; the argument for other choices of $i$ and $j$ is exactly
similar. Then from \eqref{sys2} we obtain
\begin{equation}
\zeta_2'=-2i\theta_2(x)\sqrt{\zeta_2}(\zeta_2+C_2).
\label{diffeqcasei}
\end{equation}
We know that $\zeta_2$ is not identically equal to $-C_2$
on $I$, for otherwise \eqref{mustdi}, \eqref{uzetai}, and $\zeta_1 \equiv -C_1$
 would imply that $u$ is identically equal to
$0$ on $I$, and hence also on $\mathbf R$.  Therefore, by
taking $I$ smaller if necessary, we may assume that $\zeta_2$
is never equal to $-C_2$ on $I$.

It then follows from
\eqref{diffeqcasei} that $\theta_2(x)$ is analytic on $I$.
But since $\theta_2$ takes values in $\{-1,1\}$, the only way
this can happen is if $\theta_2$ is constant on $I$.  Setting
$v=i\theta_2\sqrt{\zeta_2}$ in \eqref{diffeqcasei}, we  obtain
\begin{equation*}
\frac{v'}{C_2-v^2}= 1,
\end{equation*}
which is equation \eqref{vprime} with $C$ replaced by $C_2$.

Moreover, we also know that $C_1$ and $C_2$ must
be real, since otherwise we would have $C_2 = C_1^\ast$, and together with $\zeta_1 \equiv -C_1$
and $\zeta_1 = \zeta_2^\ast$ this would imply $\zeta_2 \equiv -C_2$ on $I$.   Since $C_1$ is
real, then $\zeta_1 \equiv -C_1$ implies that $\zeta_1$ is real, so $\zeta_2$ must also be real on $I$.
Since we also know that $\zeta_2$ is never equal to zero on $I$, it
follows from \eqref{zetasys} that $\zeta_2 <0$ on $I$. Hence $v$
is real-valued on $I$.   Therefore we can reprise the proof of
Theorem \ref{1solthm} from \eqref{vprime} onwards, replacing $C$
by $C_2$ throughout (notice that by \eqref{uzetai} we have
$u=2(\zeta_2 + C_2)=2(C_2-v^2)$ in this case as well).   The
conclusion is that $C_2 >0$ and that
\begin{equation*}
u=\frac{2C_2}{\cosh^2(\sqrt C_2 x + K)}.
\end{equation*}
In light of the remark following Theorem \ref{1solthm}, we obtain
the formula \eqref{casei} for $u$, and this must hold on the
entire line.  Thus the statement of the theorem is proved in this
case.

Therefore we may assume henceforth that \eqref{zetipcjnezero}
holds for all $i,j\in\{1,2\}$. In this case, the right-hand sides
of both equations in \eqref{sys2} are never zero on $I$, so as in
the preceding cases it follows from \eqref{sys2} that $\theta_1$
and $\theta_2$ are both analytic and hence constant functions on
$I$, with value either $-1$ or $1$. Set
\begin{equation}
\begin{aligned}
v_1 & = -i\theta_1 \sqrt{\zeta_1}\\
v_2 &=   i\theta_2\sqrt{\zeta_2},
\end{aligned}
\label{defvi}
\end{equation}
and define $\alpha_j$, for $j=1,2$, to be complex numbers such
that
\begin{equation}
\alpha_j^2 = C_j. \label{defalphaj}
\end{equation}
For definiteness we will choose $\alpha_j$ to be the square root
of $C_j$ given by $\alpha_j = |C_j|^{1/2}e^{i\theta/2}$, where
$-\pi < \theta \le \pi$ and $C_j=|C_j|e^{i\theta}$.  In
particular, this choice guarantees that if $C_1^\ast=C_2$, then
$\alpha_1^\ast=\alpha_2$.

Since $\zeta_1$ and $\zeta_2$ are nonzero on $I$, so are $v_1$
and $v_2$.  Also, as noted
 above after \eqref{zetasys}, either $\zeta_1$ and $\zeta_2$ are
 both negative at all points of $I$, or $\zeta_1^\ast = \zeta_2$
 on $I$.  In the former case, we have that $v_1$ and $v_2$ are real-valued on $I$.  In the latter case,
 we see by taking  the conjugate of the first equation in \eqref{sys2}, comparing the result to the
 second equation in \eqref{sys2}, and using the fact that $\sqrt{z^\ast}=-(\sqrt z)^\ast$ on $\Omega$, that $\theta_1 =-\theta_2$ on $I$.  Therefore from \eqref{defvi}
 we obtain that $v_1^\ast=v_2$ on $I$.

We can now rewrite \eqref{sys2} as the following system for $v_1$
and $v_2$:
 \begin{equation}
 \begin{aligned}
v_1'&=\frac{(\alpha_1^2-v_1^2)(\alpha_2^2-v_1^2)}{v_2^2-v_1^2}\\
 v_2'&=\frac{(\alpha_1^2-v_2^2)(\alpha_2^2-v_2^2)}{v_1^2-v_2^2},
 \end{aligned}
\label{sys4}
\end{equation}
where either $v_1$ and $v_2$ are both real-valued on $I$, or
$v_1^\ast = v_2$ on $I$. Choose $x_0 \in I$, and define
\begin{equation}
\begin{aligned}
V_1&=v_1(x_0)\\
V_2&=v_2(x_0).
\end{aligned}
\label{sys4init}
\end{equation}
(For future reference we note that, as values of $v_1$ and $v_2$,
$V_1$ and $V_2$ must both be nonzero, and either $V_1$ and $V_2$
both real or $V_1^\ast = V_2$.) Recalling that $v_1^2 \ne v_2^2$
on $I$, we have that the right-hand sides of the equations in
\eqref{sys4} define analytic functions of $v_1$ and $v_2$ on $I$.
Therefore the system \eqref{sys4}, together with the initial data
\eqref{sys4init}, uniquely determines $v_1$ and $v_2$ on some
neighborhood of $x_0$ on $I$. Furthermore, from $v_1$ and $v_2$
one can then recover $u$ via \eqref{uzetai} as
\begin{equation}
u(x)=2(-v_1^2(x)-v_2^2(x)+\alpha_1^2+\alpha_2^2). \label{recovufromv}
\end{equation}

We will complete the proof of Theorem \ref{2solthm} by explicitly solving the initial-value problem
\eqref{sys4} and \eqref{sys4init} for $v_1$ and $v_2$, and then showing that, of the functions $u$ which arise
from these solutions via \eqref{recovufromv}, the only ones which extend to $H^1$ functions on $\mathbf R$ are those
given by \eqref{desiredu}.

The system \eqref{sys4} can be integrated by separating the
variables $v_1$ and $v_2$.  Since \eqref{zetipcjnezero} holds on
$I$, we have that $v_i(x) \ne \alpha_j$ on $I$ for $i,j \in
\{1,2\}$.  This allows us to rewrite \eqref{sys4} in the form
\begin{equation}
\begin{aligned}
\frac{v_1'}{(\alpha_1^2-v_1^2)(\alpha_2^2-v_1^2)}+\frac{v_2'}{(\alpha_1^2-v_2^2)(\alpha_2^2-v_2^2)}&=0\\
\noalign{\vskip2mm}
\frac{-v_1^2v_1'}{(\alpha_1^2-v_1^2)(\alpha_2^2-v_1^2)}+\frac{-v_2^2v_2'}{(\alpha_1^2-v_2^2)(\alpha_2^2-v_2^2)}&=1.
\end{aligned}
\label{visyst}
\end{equation}

 To compute the solutions of
\eqref{visyst} in the cases when the quantities $\alpha_i^2$
coincide or are zero, it will be helpful to consider as well a
system in which the values of the constants $\alpha_i$ are
slightly perturbed. For any positive number $\delta$ and complex
number $z_0$, let $B_{\delta}(z_0)$ denote the open ball of radius
$\delta$ centered at $z_0$ in $\mathbf C$. Choose $\delta$ to be
any positive number such that $\delta < \frac14 \min\{|\alpha_i -
V_j|: i, j = 1,2\}$ and $\delta < |V_1-V_2|$. For each $j \in
\{1,2\}$, we define functions $G_j=G_j(\beta_1,\beta_2,v)$ and
$H_j=H_j(\beta_1,\beta_2,v)$ on $B_{\delta}(\alpha_1)\times
B_{\delta}(\alpha_2)\times B_{\delta}(V_j)$ by
\begin{equation}
\begin{aligned}
G_j(\beta_1,\beta_2,v)&=\int_{V_j}^v\frac{dw}{(\beta_1^2-w^2)(\beta_2^2-w^2)}\\
H_j(\beta_1,\beta_2,v)&=\int_{V_j}^v\frac{-w^2dw}{(\beta_1^2-w^2)(\beta_2^2-w^2)}\
.
\end{aligned}
\label{gjhj}
\end{equation}
From the definition of $\delta$ we know that the integrands in
\eqref{gjhj} are nonsingular functions of $w$ on
$B_{\delta}(V_j)$, so $G_j$ and $H_j$ are well-defined and are
analytic functions of $\beta_1$, $\beta_2$, and $v$ on their
domains,  as long as the integrals in their definition are taken
over paths from $V_j$ to $v$ which lie within $B_{\delta}(V_j)$.
We can then define $F:\mathbf C^2 \times \mathbf C^2 \times
\mathbf R \to \mathbf C^2$, with domain
$U_F=B_{\delta}(\alpha_1)\times B_{\delta}(\alpha_2)\times
B_{\delta}(V_1)\times B_{\delta}(V_2)\times \mathbf R$, by setting
\begin{equation}
F(\beta_1,\beta_2,v_1,v_2,x) =\left[\begin{matrix}
F_1\\
F_2
\end{matrix}\right]
=\left[\begin{matrix}
G_1(\beta_1,\beta_2,v_1)+G_2(\beta_1,\beta_2,v_2)\\
H_1(\beta_1,\beta_2,v_1)+H_2(\beta_1,\beta_2,v_2)-x+x_0
\end{matrix}\right].
\label{fdef}
\end{equation}

\begin{lem}
There exist numbers $\delta_1 \in (0,\delta)$ and $\delta_2 >0$
such that for each $(\beta_1,\beta_2,x)$ in
$B_{\delta_1}(\alpha_1)\times B_{\delta_1}(\alpha_2)\times
\{|x-x_0|< \delta_1\} \subseteq \mathbf C^2 \times \mathbf R$,
there is a unique pair $(v_1,v_2)$ in $B_{\delta_2}(V_2)\times
B_{\delta_2}(V_2) \subseteq \mathbf C^2$ satisfying
\begin{equation}
F(\beta_1,\beta_2,v_1,v_2,x)=0.
\label{v1v2epseqn}
\end{equation}
The functions $v_1(\beta_1,\beta_2,x)$ and
$v_2(\beta_1,\beta_2,x)$ so defined are analytic functions of
their arguments, and for each $(\beta_1,\beta_2) \in
B_{\delta_1}(\alpha_1)\times B_{\delta_1}(\alpha_2)$, are
solutions of the system of ordinary differential equations
\begin{equation}
\begin{aligned}
\frac{v_1'}{(\beta_1^2-v_1^2)(\beta_2^2-v_1^2)}+\frac{v_2'}{(\beta_1^2-v_2^2)(\beta_2^2-v_2^2)}&=0\\
\noalign{\vskip2mm}
\frac{-v_1^2v_1'}{(\beta_1^2-v_1^2)(\beta_2^2-v_1^2)}+\frac{-v_2^2v_2'}{(\beta_1^2-v_2^2)(\beta_2^2-v_2^2)}&=1,
\end{aligned}
\label{visystbeta}
\end{equation}
on $\{|x-x_0|<\delta_1\}$, with initial conditions
\begin{equation}
\begin{aligned}
v_1(\beta_1,\beta_2,x_0)&=V_1\\
v_2(\beta_1,\beta_2,x_0)&=V_2.
\end{aligned}
\label{viinitcond}
\end{equation}
\label{impfnv}
\end{lem}

\begin{proof}  A computation of the Jacobian of $F$ with respect to
the variables $v_1$ and $v_2$, reveals that, at all points
$P=(\beta_1,\beta_2,v_1,v_2,x)$ in the domain $U_F$ of $F$, we
have
\begin{equation}
\begin{aligned}
\nabla_{v_1,v_2} F(P)&= \left.\left[
\begin{matrix}
\partial F_1/\partial v_1 & \partial F_1/\partial v_2 \\
\partial F_2/\partial v_1 & \partial F_2/\partial v_2
\end{matrix}
\right]\right|_P\\
 & =
\begin{bmatrix}
\dfrac{1}{(\beta_1^2 - v_1^2)(\beta_2^2 - v_1^2)} &
\dfrac{1}{(\beta_1^2 - v_2^2)(\beta_2^2 - v_2^2)}\\[12pt]
\dfrac{-v_1^2}{(\beta_1^2 - v_1^2)(\beta_2^2 - v_1^2)} &
\dfrac{-v_2^2}{(\beta_1^2 - v_2^2)(\beta_2^2 - v_2^2)}
\end{bmatrix}
  .
  \end{aligned}
  \label{jacf}
\end{equation}
The determinant of the matrix in \eqref{jacf} is
$(v_1^2-v_2^2)/\prod_{i,j=1,2}(\beta_i^2 - v_j^2)$,
and is therefore nonzero for all $P\in U_F$.  In particular, when
$P_0=(\alpha_1,\alpha_2,V_1,V_2,x_0)$ we have that
$\nabla_{v_1,v_2} F(P_0)$ is an invertible map from $\mathbf C^2$
to $\mathbf C^2$. Moreover, $F(P_0)=0$.  The assertions of the
Lemma concerning the existence, uniqueness, and analyticity of the
functions $v_1$ and $v_2$ which satisfy \eqref{v1v2epseqn}
therefore follow from the Implicit Function Theorem (cf.\ \S 15 of
\cite{De}). Equations \eqref{visystbeta} then follow by
differentiating both sides of \eqref{v1v2epseqn} with respect to
$x$.  The initial conditions \eqref{viinitcond} are a consequence
of the uniqueness assertion for the $v_i$ and the fact that
\begin{equation*}
F(\beta_1,\beta_2,V_1,V_2,x_0)=0
\end{equation*}
for each $(\beta_1,\beta_2) \in B_{\delta_1}(\alpha_1)\times
B_{\delta_1}(\alpha_2)$.
\end{proof}

Motivated by \eqref{recovufromv}, we now define, for each
$(\beta_1,\beta_2,x) \in B_{\delta_1}(\alpha_1)\times
B_{\delta_1}(\alpha_2)\times\{|x-x_0|<\delta_1\}$,
\begin{equation}
u(\beta_1,\beta_2,x)=2(-v_1(\beta_1,\beta_2,x)^2-v_2(\beta_1,\beta_2,x)^2+\beta_1^2+\beta_2^2).
\label{ub1b2}
\end{equation}

\begin{cor} The solution $u$ of \eqref{stathier2} described in
Theorem \ref{2solthm} is related to the functions
$u(\beta_1,\beta_2,x)$ by
\begin{equation}
u(x)=u(\alpha_1,\alpha_2,x)=\lim_{(\beta_1,\beta_2) \to
(\alpha_1,\alpha_2)} u(\beta_1,\beta_2,x) \label{u0limueps}
\end{equation} for all
$x$ such that $|x-x_0|<\delta_1$.
\label{ualphaandbeta}
\end{cor}

\begin{proof} By Lemma \ref{impfnv}, the functions $v_1(\alpha_1,\alpha_2,x)$ and $v_2(\alpha_1,\alpha_2,x)$
satisfy \eqref{visyst} for $|x-x_0|< \delta_1$, and therefore, since $\nabla_{v_1,v_2}F(P_0)$ is invertible, also satisfy
\eqref{sys4}.  Comparing the initial conditions \eqref{viinitcond}
and \eqref{sys4init}, we see from the uniqueness of the solutions
of the initial value problem for \eqref{sys4} that
$(v_1(x),v_2(x))=(v_1(\alpha_1,\alpha_2,x),v_2(\alpha_1,\alpha_2,x))$ for all $x$ in
some neighborhood of $x_0$.  Putting
$(\beta_1,\beta_2)=(\alpha_1,\alpha_2)$ in \eqref{ub1b2} and
comparing with \eqref{recovufromv}, we see that $u(x)$ then agrees
with $u(\alpha_1,\alpha_2,x)$ on this neighborhood.  Since $u(x)$
is analytic on $\mathbf R$ and $u(\alpha_1,\alpha_2,x)$ is
analytic on $\{|x-x_0|<\delta_1\}$, they must agree on all of
$\{|x-x_0|<\delta_1\}$.  Finally, the assertion that $u(x)$ is
given by the limit in \eqref{u0limueps} follows from the fact that
$u(\beta_1,\beta_2,x)$ is analytic and hence continuous in
$\beta_1$ and $\beta_2$.
\end{proof}

The next step in the proof of Theorem \ref{2solthm} is to
explicitly determine the functions $u(\beta_1,\beta_2,x)$ defined
in \eqref{ub1b2}. Generically, we will have that $\beta_1^2$ and
$\beta_2^2$ are distinct and nonzero, even if this is not true
when $\beta_1=\alpha_1$ and $\beta_2 = \alpha_2$. In this case, we
have the following Lemma.

\begin{lem} Suppose $(\beta_1,\beta_2,x) \in
B_{\delta_1}(\alpha_1)\times B_{\delta_1}(\alpha_2)\times
\{|x-x_0|< \delta_1\}$, and suppose that $\beta_1^2 \ne
\beta_2^2$, and $\beta_j \ne 0$ for $j=1,2$.  Then
\begin{equation}
u(\beta_1,\beta_2,x)=2(D'/D)', \label{soln1}
\end{equation}
where
\begin{equation}
D=D(\beta_1,\beta_2,x)=\left|
\begin{matrix}
y_1 & y_2\\
y_1' & y_2'\end{matrix} \right|, \label{dneq2}
\end{equation}
with
\begin{equation}
y_j=\sinh (\beta_j(x-x_0) + M_j) \label{yjsinhaj}
\end{equation}
for $j=1,2$.  Here  $M_1$ and $M_2$ can be taken to be any complex
numbers satisfying
\begin{equation}
e^{2M_j}=\frac{(\beta_j+V_1)(\beta_j+V_2)}{(\beta_j-V_1)(\beta_j-V_2)}
\label{defmj}
\end{equation}
for $j=1,2$. \label{gensol}
\end{lem}

\begin{proof}
Under the stated assumptions on $(\beta_1,\beta_2,x)$, we have,
for $v \in B_{\delta}(V_j)$,
\begin{equation}
\begin{aligned}
G_j(\beta_1,\beta_2,v)&=\frac{1}{\beta_2^2-\beta_1^2}\left[L_{\beta_1,V_j}(v)-L_{\beta_2,V_j}(v)\right]\\
\noalign{\vskip2mm}
H_j(\beta_1,\beta_2,z)&=\frac{1}{\beta_2^2-\beta_1^2}\left[-\beta_1^2L_{\beta_1,V_j}(v)+\beta_2^2L_{\beta_2,V_j}(v)
\right],
\end{aligned}
\label{defGH}
\end{equation}
where $L_{\beta,V}$ is defined in \eqref{defG0}. Let $v_j$ denote
$v_j(\beta_1,\beta_2,x)$ for $j=1,2$, and define
\begin{equation}
w_j = 2 \beta_j L_{\beta_j,V_1}(v_1) + 2 \beta_j L_{\beta_j,
V_2}(v_2). \label{defwj}
\end{equation}
Substituting \eqref{defGH} into \eqref{fdef}, we see that
\eqref{v1v2epseqn} can be rewritten as  a linear system for $w_1$
and $w_2$:
\begin{equation*}
\begin{aligned}
\beta_2 w_1 - \beta_1 w_2 & = 0\\
-\beta_1 w_1 + \beta_2 w_2 &= 2(\beta_2^2 - \beta_1^2) (x-x_0).
\end{aligned}
\end{equation*}
Since $\beta_1 \ne \beta_2$, the system has a unique solution, given by
\begin{equation}
w_j=2\beta_j (x-x_0)\quad \text{for $j=1,2$.} \label{solnforwi}
\end{equation}

Now from \eqref{eL} and \eqref{defwj} we have that
\begin{equation*}
e^{w_j} = \left(\frac{\beta_j + v_1}{\beta_j -
v_1}\right)\left(\frac{\beta_j + v_2}{\beta_j - v_2}\right)
\left(\frac{\beta_j - V_1}{\beta_j +
V_1}\right)\left(\frac{\beta_j - V_2}{\beta_j + V_2}\right).
\end{equation*}
Therefore after exponentiating both sides of \eqref{solnforwi},
we obtain, for $j=1,2$,
\begin{equation}
 \frac{(\beta_j+v_1)(\beta_j+v_2)}{(\beta_j-v_1)(\beta_j-v_2)}
 = \frac{(\beta_j+V_1)(\beta_j+V_2)}{(\beta_j-V_1)(\beta_j-V_2)}\ e^{2\beta_j(x-x_0)}=e^{2(\beta_j(x-x_0)+M_j)},
\label{defP12}
\end{equation}
where $M_1$ and $M_2$ are any complex numbers such that
\eqref{defmj} holds.

If we now define
\begin{equation*}
\begin{aligned}
W_1&=-(v_1+v_2) \\
W_0&=v_1v_2,
\end{aligned}
\end{equation*}
then \eqref{defP12} can be rewritten as
\begin{equation}
\left[
\begin{matrix}
y_1 & y_1'\\
y_2 & y_2'\end{matrix} \right] \left[\begin{matrix} W_0 \\
W_1\end{matrix}\right]=-\left[\begin{matrix} y_1'' \\
y_2''\end{matrix}\right], \label{mateqnforWi}
\end{equation}
where $y_j$ is given by \eqref{yjsinhaj} for $j=1,2$. Solving \eqref{mateqnforWi} by
Cramer's rule, we find that $W_1$ and $W_0$ are given by
\eqref{defWi} and \eqref{defW0}, with $N=2$. Therefore we can
conclude from Corollary \ref{nicecor} that
\begin{equation}
W_1' = W_1^2-2W_0-\beta_1^2-\beta_2^2. \label{w0w1}
\end{equation}

On the other hand, from \eqref{ub1b2} we have
\begin{equation}
u(\beta_1,\beta_2,x)=2(2W_0-W_1^2+\beta_1^2+\beta_2^2),
\label{recoveru}
\end{equation}
 and from \eqref{w0w1} and \eqref{recoveru} it then follows that
$u(\beta_1,\beta_2,x)=-2W_1'$.  By \eqref{defWi}, we have $W_1' = -D'/D$, with $D$ given by
\eqref{dneq2}.  Then \eqref{soln1} follows.
\end{proof}

We now determine the conditions on $\beta_1$ and $\beta_2$ under
which \eqref{soln1} defines a nonsingular solution on $\mathbf R$,
or in other words, under which $D(\beta_1,\beta_2,x)$ has no
zeroes on $\mathbf R$.

\begin{lem}   Suppose $(\beta_1,\beta_2) \in
B_{\delta_1}(\alpha_1)\times B_{\delta_1}(\alpha_2)$, and suppose
that $\beta_1^2 \ne \beta_2^2$, and $\beta_j \ne 0$ for $j=1,2$.
Suppose also that $0 \le \Re \beta_1 \le \Re \beta_2$, and either
$\beta_1^2$ and $\beta_2^2$ are both real numbers, or
$\beta_1^\ast=\beta_2$.

If, as a function of $x$, $u(\beta_1,\beta_2,x)$ can be
analytically continued to an analytic function on the entire real
line, then  $0 < \beta_1 < \beta_2$, and there exist numbers $a_1$
and $a_2$ with $a_2 < 0 < a_1$ such that
\begin{equation*}
u(\beta_1, \beta_2, x) = \psi^{(2)}(x; a_1, a_2; \beta_1, \beta_2)
\end{equation*}
on $\mathbf R$.
 \label{uis2sol}
\end{lem}

\begin{proof} Observe that, according to Lemma \ref{gensol}, if $u(\beta_1,\beta_2,x)$
can be analytically continued to an analytic function on the real
line, then the function $D(\beta_1,\beta_2,x)$ defined in
\eqref{dneq2}--\eqref{defmj} must be nonzero for all $x \in
\mathbf R$.  In fact, from \eqref{soln1} it is easy to see that,
at any point $x$ where $D(\beta_1,\beta_2,x)$ has a zero,
$u(\beta_1,\beta_2,x)$ will have a pole of order two.

If $\beta_1^\ast = \beta_2$, then since either $V_1$ and $V_2$ are
both real or $V_1^\ast = V_2$, we see from \eqref{defmj} that we
can choose $M_1$ and $M_2$ so that $M_1^\ast=M_2$. Define $K_j =
-\beta_j x_0 + M_j$ and
\begin{equation*}
A_j = \beta_j(x-x_0)+M_j = \beta_j x + K_j,
\end{equation*}
so that $A_1^\ast = A_2$ and, by \eqref{yjsinhaj}, $y_j= \sinh
A_j$ for $j= 1,2$.  Let $a=\Re(\beta_1)$, $b=\Im(\beta_1)$,
$P=\Re(K_1)$, and $Q=\Im(K_1)$. Then from \eqref{dneq2}, we have
\begin{equation*}
\begin{aligned}
D&=\beta_2\sinh(A_1)\cosh(A_2)-\beta_1\cosh(A_1)\sinh(A_2)\\
&= \beta_1^\ast\sinh(A_1)\cosh(A_1^\ast)-\beta_1\cosh(A_1)\sinh(A_1^\ast)\\
&=2\Im \left(\beta_1^\ast \sinh(A_1)\cosh(A_1^\ast)\right)\\
&=a\sin(2(bx+Q))-b\sinh(2(ax+P)).
\end{aligned}
\end{equation*}
Therefore $D$ changes sign as $x$ goes from large negative values
to large positive values, and hence $D$ must equal zero for some
$x \in \mathbf R$.  Therefore, as remarked above,
$u(\beta_1,\beta_2,x)$ cannot be continued to an analytic function
on $\mathbf R$.

Next, suppose $\beta_1^2$ and $\beta_2^2$ are real and of opposite
sign, say $\beta_1^2 < 0 < \beta_2^2$. In this case $\beta_1$ is
purely imaginary, say $\beta_1 = ib$, and $\beta_2$ is real.  Then
when $M_1$ and $M_2$ are chosen so that \eqref{defmj} holds, it
follows that $(e^{2M_1})^\ast=e^{-2M_1}$ and
$(e^{2M_2})^\ast=e^{2M_2}$. We can thus take $M_1$ to be purely
imaginary, and $M_2$ so that either $M_2$ is real or $\Im M_2 =
\pi/2$. Now the same arguments as in equations \eqref{e2areal1} to
\eqref{pureimconj2} in Section \ref{sec:1solconverse} show that
\begin{equation}
y_1=i\sin(bx+K_1)
\label{2negy1}
\end{equation}
for some real $K_1$, and either
\begin{equation}
y_2=\sinh(\beta_2 x +K_2) \label{2negy2a}
\end{equation}
or
\begin{equation}
y_2=i\cosh(\beta_2 x + K_2) \label{2negy2b}
\end{equation}
for some real $K_2$. If \eqref{2negy1} and \eqref{2negy2a} hold, the equation $D=0$
becomes
\begin{equation*}
\frac{\tan(bx+K_1)}{b}=\frac{\tanh(\beta_2x+K_2)}{\beta_2},
\end{equation*}
which  has infinitely many solutions on $\mathbf R$. If, on the
other hand, \eqref{2negy1} and \eqref{2negy2b} hold, then the equation $D=0$ becomes
\begin{equation*}
\frac{\tan(bx+K_1)}{b}=\frac{\coth(\beta_2x+K_2)}{\beta_2},
\end{equation*}
which again has infinitely many solutions.

Now suppose that $\beta_1^2$ and $\beta_2^2$ are both negative,
say $\beta_1^2 < \beta_2^2 < 0$. Then $\beta_1$ and $\beta_2$ are
both purely imaginary, say $\beta_j = ib_j$ for $j=1,2$, and in \eqref{defmj}
we can choose $M_1$ and $M_2$ to be purely imaginary.  Thus
\begin{equation*}
y_j=i\sin(b_j x+K_j)
\end{equation*}
on $I$ for $j=1,2$, where $K_1$ and $K_2$ are real, so $D$ has
a zero at any point $x \in \mathbf R$ which satisfies the equation
\begin{equation}
\frac{\tan(b_1x+K_1)}{b_1}=\frac{\tan(b_2x+K_2)}{b_2}.
\label{twotans}
\end{equation}
Since $\beta_1^2 \ne \beta_2^2$, then $b_1 \ne b_2$, in which case
it is easy to see that equation \eqref{twotans} always has
solutions.

We have now shown that, under the stated assumptions on $\beta_1$
and $\beta_2$, $u(\beta_1,\beta_2,x)$ can be continued to an
analytic function on $\mathbf R$ only if $\beta_1^2$ and
$\beta_2^2$ are both positive, with therefore $0 < \beta_1 <
\beta_2$. In this case, when $M_1$ and $M_2$ are chosen to satisfy
\eqref{defmj}, we will have that $(e^{2M_j})^\ast=e^{2M_j}$ for
$j=1,2$.  Then, as in \eqref{2negy2a} and \eqref{2negy2b}, for
each $j$ we have that either
\begin{equation*}
y_j = \sinh(\beta_j x + K_j)
\end{equation*}
or
\begin{equation*}
y_j = i \cosh(\beta_j x + K_j),
\end{equation*}
where $K_1$ and $K_2$ are real.  There are therefore four cases to consider, which, according to
Definition \ref{defpsin}, can be summarized as follows:
\begin{equation*}
u=\psi^{(2)}(x;a_1,a_2;\beta_1,\beta_2),
\end{equation*}
where $a_j$ are nonzero numbers given by $a_j = \pm
e^{-2K_j}$ for $j=1,2$.

If $a_1$ and $a_2$ are both positive, then $D=0$ holds at any
point $x$ where
\begin{equation}
\beta_2 \tanh(\beta_2 x + K_2)-\beta_1 \tanh(\beta_1 x +K_1)=0,
\label{a1a2pos}
\end{equation}
and since the left-hand side of \eqref{a1a2pos} changes sign from
negative to positive as $x$ increases from large negative values
to large positive values, there must exist solutions to
\eqref{a1a2pos} in $\mathbf R$. Similarly, if $a_1$ and $a_2$ are
both negative, then $D=0$ when
\begin{equation*}
\beta_1 \tanh(\beta_2 x + K_2)-\beta_2 \tanh(\beta_1 x +K_1)=0,
\end{equation*}
which again must have at least one solution in $\mathbf R$.  If
$a_1<0<a_2$, then $D=0$  when
\begin{equation*}
\tanh(\beta_1 x+K_1)\tanh(\beta_2 x + K_2) =
\frac{\beta_1}{\beta_2}
\end{equation*}
which must have a solution, since the fraction on the right-hand
side is between 0 and 1, and the function on the left-hand side
attains the value zero at the points $x=-K_j/\beta_j$, $j=1,2$,
and approaches 1 as $x \to \pm \infty$.

Finally, if $a_2 < 0 < a_1$, then $D = 0$ only if
\begin{equation}
\tanh(\beta_1x+K_1)\tanh(\beta_2 x + K_2) =
\frac{\beta_2}{\beta_1}. \label{a1nega2pos}
\end{equation}
But \eqref{a1nega2pos} has no solutions, since the fraction on the
right-hand side is greater than 1, and the function on the
left-hand side always takes values less than 1. (Alternatively, in
this final case we could deduce from Lemma \ref{condforDpos} that
$D$ is never 0 on $\mathbf R$.)
\end{proof}

\subsection{The nondegenerate case}
\label{subsec: nondegencase}

In this subsection we consider the case in which the roots $C_1$
and $C_2$ of \eqref{c1c2} are nondegenerate: that is, when $C_1
\ne C_2$, $C_1 \ne 0$, and $C_2 \ne 0$. Recall that, as mentioned
following equation \eqref{basiceqn}, we can assume that either
$C_1$ and $C_2$ are both real, with $C_1 < C_2$, or $C_1^\ast =
C_2$. From our definition of $\alpha_j$ in \eqref{defalphaj} and
the remarks following, we have then that $\alpha_1$ and $\alpha_2$
are distinct and both nonzero, with $0 \le \Re \alpha_1 \le \Re
\alpha_2$; and either $\alpha_1^2$ and $\alpha_2^2$ are both real,
with $\alpha_1^2 < \alpha_2^2$, or $\alpha_1^\ast=\alpha_2$.

According to Corollary \ref{ualphaandbeta},
$u(x)=u(\alpha_1,\alpha_2,x)$ on some neighborhood of $x_0$, and
therefore $u(\alpha_1,\alpha_2,x)$ can be analytically continued
to a function which is analytic on all of $\mathbf R$.  It then
follows from Lemma \ref{uis2sol} that $0 < \alpha_1 < \alpha_2$,
and
\begin{equation*}
u(x)=\psi^{(2)}(x;a_1,a_2;\alpha_1,\alpha_2)
\end{equation*}
for some numbers $a_1$ and $a_2$ with $0<a_1<a_2$.

We have therefore proved Theorem \ref{2solthm} in the
nondegenerate case, by showing that the only possible solutions
in this case are given by \eqref{desiredu}.

To complete the proof of Theorem \ref{2solthm}, it remains to show
that the degenerate cases, when $C_1$ and $C_2$ can coincide or
vanish, cannot arise under the assumption that \eqref{stathier2}
has a nontrivial solution $u(x)$ in $H^2(\mathbf R)$.  Since any
$H^2$ solution must be analytic on $\mathbf R$, to accomplish this
it is enough to show that in the degenerate cases, any locally
analytic solution of \eqref{stathier2} extends analytically to a
function with a singularity on $\mathbf R$.

\subsection{The degenerate case $C_1 = C_2 \ne 0$}
\label{subsec: doublenonzeroroot}

In this subsection we consider the case when equation \eqref{c1c2}
has a nonzero double root, so that $C_1=C_2 =C\ne 0$ in
\eqref{visyst}.  In this case, $C$ must be real.  Define $\alpha$
to be a square root of $C$, following the convention for choice of
square roots set after \eqref{defalphaj}. From Corollary
\ref{ualphaandbeta}  we have that
\begin{equation}
u(x)=u(\alpha,\alpha,x)=\lim_{\epsilon \to
0}u(\alpha,\alpha+\epsilon,x) \label{uceqlim}
\end{equation}
for $|x-x_0|< \delta_1$.

We now compute the limit in \eqref{uceqlim}.  Since $\alpha \ne
0$, by taking $\epsilon$ positive and sufficiently small, we may
assume that $\beta_1=\alpha$ and $\beta_2 =\alpha +\epsilon$
satisfy the hypotheses of Lemma \ref{gensol}. We thus obtain that,
for $|x-x_0|<\delta_1$,
\begin{equation}
u(\alpha,\alpha+\epsilon,x)=2(D(\alpha,\alpha+\epsilon,x)'/D(\alpha,\alpha+\epsilon,x))'
\label{uepssoln}
\end{equation}
where for all complex numbers $s$ and $t$ we define $D(s,t,x)$ by
\begin{equation}
D(s,t,x)=t\sinh A(s,x)\cosh A(t,x) - s \sinh
A(t,x)\cosh A(s,x),
\label{defDst}
\end{equation}
with
\begin{equation*}
A(s,x)=s(x-x_0)+M(s)
\end{equation*}
and
\begin{equation}
e^{M(s)}=\frac{(s+V_1)(s+V_2)}{(s-V_1)(s-V_2)}. \label{ems}
\end{equation}
Since the right side of \eqref{ems} is nonzero for $s=\alpha$, we
may assume that $M(s)$ is defined and analytic for $s$ in some
neighborhood of $\alpha$.

Observe that for $\epsilon >0$ we can rewrite equation \eqref{uepssoln} in the form
\begin{equation}
u(\alpha,\alpha+\epsilon,x)=2(D_1(\epsilon,x)'/D_1(\epsilon,x))',
\label{uepsxd}
\end{equation}
where
\begin{equation*}
D_1(\epsilon,x)=\frac{D(\alpha,\alpha+\epsilon,x)}{\epsilon}=\frac{D(\alpha,\alpha+\epsilon,x)-D(\alpha,\alpha,x)}{\epsilon}.
\end{equation*}
Let us define
\begin{equation*}
D_1(0,x) =  \lim_{\epsilon \to
0}D_1(\epsilon,x)=\left. \frac{\partial }{\partial
\epsilon}\right|_{\epsilon=0}D(\alpha,\alpha+\epsilon,x).
\end{equation*}
 Computing the derivative gives
\begin{equation}
  D_1(0,x)
=\frac12\sinh (2\alpha(x-x_0)+2M(\alpha)) -\alpha(x-x_0)-\alpha
M'(\alpha)
\label{tildedzx}
\end{equation}
for all $x \in \mathbf R$.  Choose $I_1$ to be any nonempty
subinterval of $(x_0 - \delta_1,x_0 + \delta_1)$ such that
$D_1(0,x) \ne 0$ on the closure of $I_1$.  Then, as $\epsilon$
goes to $0$, $D_1(\epsilon,x)$ will converge to $D_1(0,x)$
uniformly on $I_1$, and the derivatives with respect to $x$ of
$D_1(\epsilon,x)$ will converge to the corresponding derivatives
of $D_1(0,x)$ uniformly on $I_1$ as well.  It then follows from
\eqref{uceqlim} and \eqref{uepsxd} that
\begin{equation}
u(x) = 2(D_1(0,x)'/D_1(0,x))' \label{ubarsoln}
\end{equation}
for $x \in I_1$.

We now show that \eqref{ubarsoln} extends analytically to a
singular function on $\mathbf R$. For this it is enough to show
that $D_1(0,x)$ has at least one zero on $\mathbf R$.

If $C$ is positive, then $\alpha$ is a positive real number.
Recalling that either $V_1$ and $V_2$ are both real or $V_1^\ast =
V_2$, we see that $e^{M(\alpha)}$ is real.  Depending on whether
the right-hand side of \eqref{ems} is positive or negative at
$s=\alpha$, we can choose $M(\alpha)$ to either be real, or to
have imaginary part $\pi/2$.  Also since \eqref{ems} implies
\begin{equation}
M'(s) = \frac{(s+V_1)+(s+V_2)}{(s+V_1)(s+V_2)} -
\frac{(s-V_1)+(s-V_2)}{(s-V_1)(s-V_2)},
\label{mprime}
\end{equation}
$M'(\alpha)$ is real in any case.

If $M(\alpha)$ is real, from \eqref{tildedzx} we see that
$D_1(0,x)
>0$ for $x$ large and positive, and $D_1(0,x)<0$ for $x$ large and
negative. Therefore there must exist at least one $x \in \mathbf
R$ for which $D_1(0,x)=0$.  On the other hand, if
$M(\alpha)=K+i\pi/2$ for $K$ real, then \eqref{tildedzx} gives
\begin{equation*}
  D_1(0,x)=-\frac12\sinh (2\alpha(x-x_0)+2K) -\alpha(x-x_0)-\alpha
M'(\alpha),
\end{equation*}
and again we see that $D_1(0,x)=0$ for some $x \in \mathbf R$.

If $C$ is negative, then $\alpha = i \gamma$ for some positive
number $\gamma$. We see from \eqref{ems} that $e^{M(i\gamma)}$ has
modulus 1, and so $M(i\gamma)$ can be taken to be purely
imaginary: say $M(i\gamma)=i\eta_1$ for $\eta_1 \in \mathbf R$.
From \eqref{mprime} it is readily checked that $M'(i\gamma)^\ast =
M'(i\gamma)$, and hence $M'(i \gamma)$ is real: say
$M'(i\gamma)=\eta_2$ for $\eta_2 \in \mathbf R$.  Then from
\eqref{tildedzx} we have that
\begin{equation}
D_1(0,x) =i \left(\frac12
\sin(2\gamma(x-x_0)+2\eta_1)-\gamma(x-x_0)-\gamma\eta_2\right).
\label{dtildcneg}
\end{equation}
The quantity in parentheses in \eqref{dtildcneg} is a real-valued
function of $x$ which is positive for large negative values of
$x$, and is negative for large positive values of $x$. Therefore
$D_1(0,x)$ must equal zero for some $x \in \mathbf R$.

We have shown, then, that whenever $C_1 = C_2 \ne 0$, $u(x)$ is given on
some open interval by the function on the right hand side of \eqref{ubarsoln}, which however
 cannot be extended
analytically to a function on all of $\mathbf R$.  This contradicts
the fact that $u(x)$ is analytic on $\mathbf R$.   Thus no
nontrivial $H^2$ solutions to \eqref{stathier2} can exist in this case.

\subsection{The degenerate case when $C_1=0$ or $C_2=0$ (but not both)} Here we consider
the cases when either $0=C_1<C_2$ or $C_1<C_2=0$.

Suppose first that $C_1=0$ and $C_2>0$, and let
$\alpha=\sqrt{C_2}$.  From Corollary \ref{u0limueps} we have that
\begin{equation}
u(x)=u(0,\alpha,x)=\lim_{\epsilon \to 0}u(\epsilon,\alpha,x)
\label{uc10lim}
\end{equation}
for $|x-x_0|<\delta_1$.  For $\epsilon$ sufficiently small, we can
take $\beta_1=\epsilon$ and $\beta_2=\alpha$ in Lemma
\ref{gensol}, obtaining
\begin{equation}
u(\epsilon,\alpha,x)=2(D(\epsilon,\alpha,x)'/D(\epsilon,\alpha,x))',
\label{uepsalpha}
\end{equation}
where $D$ is as defined in \eqref{defDst}.   Again, we may assume that $M(s)$ in
\eqref{ems} is defined and
analytic for $s$ in some neighborhood of $\alpha$ and for $s$ in
some neighborhood of $0$.  Note that we can take
$M(0)=0$, so $A(0,x)=0$.

As in \eqref{uepsxd}, we can write, for all $\epsilon >0$,
\begin{equation}
u(\epsilon,\alpha,x)=2(D_2(\epsilon,x)'/D_2(\epsilon,x))',
\label{uepsalphd}
\end{equation}
where
\begin{equation*}
D_2(\epsilon,x)=\frac{D(\epsilon,\alpha,x)}{\epsilon}=\frac{D(\epsilon,\alpha,x)-D(0,\alpha,x)}{\epsilon}
\end{equation*}
and $D(0,\alpha,x)=0$.  Define
\begin{equation*}
D_2(0,x)=\lim_{\epsilon \to 0}D_2(\epsilon,x)=\left.
\frac{\partial }{\partial
\epsilon}\right|_{\epsilon=0}D(\epsilon,\alpha,x).
\end{equation*}
We find by differentiating that
\begin{equation}
D_2(0,x)=\alpha
\cosh(\alpha(x-x_0)+M(\alpha))\left[x-x_0+M'(0)\right]-\sinh(\alpha(x-x_0)+M(\alpha)).
\label{d20x}
\end{equation}
As in Subsection \ref{subsec: doublenonzeroroot}, on any
subinterval of $\{|x-x_0|<\delta_1\}$ where $D_2(0,x)$ is bounded
away from zero, it follows from \eqref{uc10lim} and
\eqref{uepsalphd} that
\begin{equation}
u(x) =u(0,\alpha,x)= 2(D_2(0,x)'/D_2(0,x))'. \label{ud2}
\end{equation}

Now $D_2(0,x)=0$ at any point $x$ for which
\begin{equation}
\alpha(x-x_0+M'(0))=\tanh(\alpha(x-x_0)+M(\alpha)). \label{d2xe0}
\end{equation}  Again as in Subsection \ref{subsec: doublenonzeroroot}, since $\alpha$
is real then $M'(\alpha)$ is real and $M(\alpha)$ can be taken to
either be real or to have imaginary part $\pi/2$.  Since
$\tanh(x+i\pi/2)=\tanh(x)$ for all $x \in \mathbf R$, in either
case \eqref{d2xe0} takes the form
\begin{equation*}
\alpha x-\beta_1 = \tanh(\alpha x -\beta_2),
\end{equation*}
where $\beta_1$ and $\beta_2$ are real numbers, and hence must
have a solution at some point in $\mathbf R$.  It follows then
from \eqref{ud2} that $u$ cannot be analytically continued to all
of $\mathbf R$.

There remains to consider the case when $C_1<0$ and $C_2=0$.  Let
$\alpha = i\beta$, where $\beta>0$ and $\beta^2=|C_1|$. From
Corollary \ref{u0limueps} we have
\begin{equation*}
u(x)=u(\alpha,0,x)=\lim_{\epsilon \to 0} u(\alpha,\epsilon,x),
\end{equation*}
and Lemma \eqref{gensol} gives, for $\epsilon$ sufficiently small,
\begin{equation*}
u(\alpha,\epsilon,x)=2(D(\alpha,\epsilon,x)'/D(\alpha,\epsilon,x))'=2(D(\epsilon,\alpha,x)'/D(\epsilon,\alpha,x))',
\end{equation*}
provided $|x-x_0|<\delta_1$, where $D$ is still given by \eqref{defDst}.  Hence
$u(\alpha,\epsilon,x)$ is still given by the right-hand side of \eqref{uepsalpha}, and the calculations in
\eqref{uepsalphd} to \eqref{d2xe0} apply to $u(\alpha,\epsilon,x)$, the only difference being that
now $\alpha=i\beta$ is purely imaginary.  In this case, as seen in
Subsection \ref{subsec: doublenonzeroroot}, $M(\alpha)$ can be
taken to be purely imaginary, and so \eqref{d2xe0} can be
rewritten in the form
\begin{equation*}
\beta x - \beta_1 = \tan(\beta x - \beta_2),
\end{equation*}
where $\beta_1$ and $\beta_2$ are real numbers.  This equation has
(infinitely many) real solutions, and so again $u$ cannot be
analytically continued to all of $\mathbf R$.

We have thus proved that when either $C_1$ or $C_2$ (but not both)
is zero, then $u$ cannot be analytically continued to $\mathbf R$.
So under the assumption that $u$ is an $H^2$ solution of
\eqref{stathier2}, this case cannot arise.

\subsection{The degenerate case $C_1=C_2=0$}

Finally we consider the case when $C_1$ and $C_2$ are both zero.
Then, by Corollary \eqref{ualphaandbeta},
\begin{equation}
u(x)=u(0,0,x)=\lim_{\alpha \to 0} u(0,\alpha,x), \label{u00x}
\end{equation}
where $u(0,\alpha,x)$ is given for $\alpha > 0$ by \eqref{ud2}
with \eqref{d20x}.  To emphasize the dependence of $D_2(0,x)$ on
$\alpha$, let us denote $D_2(0,x)$ by $D_3(\alpha,x)$ in what
follows.  That is,
\begin{equation*}
D_3(\alpha,x)=\alpha
\cosh(\alpha(x-x_0)+M(\alpha))\left[x-x_0+M'(0)\right]-\sinh(\alpha(x-x_0)+M(\alpha)).
\end{equation*}

The limit in \eqref{u00x} is more singular than those in preceding
sections, because $D_3(\alpha,x)$ has a zero of order three at
$\alpha=0$.  That is, we have
\begin{equation*}
D_3(0,x)=\frac{\partial D_3}{\partial
\alpha}(0,x)=\frac{\partial^2 D_3}{\partial \alpha^2}(0,x) = 0.
\end{equation*}
Therefore, to obtain a formula for $u(0,0,x)$, we should define
\begin{equation*}
D_4(\alpha,x)=\frac{D_3(\alpha,x)}{\alpha^3},
\end{equation*}
and we will have
\begin{equation*}
u(0,0,x)=2(D_4(0,x)'/D_4(0,x))',
\end{equation*}
where
\begin{equation}
D_4(0,x)=\lim_{\alpha \to 0}D_4(\alpha,x)=\frac16
\left.\frac{\partial^3}{\partial \alpha^3}\right|_{\alpha =
0}D_3(\alpha,x). \label{d40x}
\end{equation}
An elementary but fairly tedious computation of the derivative in
\eqref{d40x} shows that
\begin{equation*}
D_4(0,x)=\frac13 (x-x_0 + M'(0))^3 -\frac16 M'''(0),
\end{equation*}
and as $M'(s)$ is real for all real $s$ by \eqref{mprime}, we have
that $M'(0)$ and $M'''(0)$ are real.  Clearly then $D_4(0,x)$ has
a zero at some $x \in \mathbf R$, and so $u(x)=u(0,0,x)$ cannot be
extended to an analytic function on $\mathbf R$.

To summarize, we have now shown that in all the degenerate cases,
when $C_1$ or $C_2$ are zero or when $C_1=C_2$, \eqref{stathier2}
cannot have an $H^2$ solution on $\mathbf R$; and in the
nondegenerate case the only possible solutions are given by
\eqref{desiredu}. This then completes the proof of Theorem
\ref{2solthm}.
\end{proof}

\section{The stationary equation for general $N$} \label{sec:Nsolconverse}

We conclude with a few comments as to how the results above may be
generalized to arbitary stationary equations of the KdV hierarchy.
In view of the first remark following Definition \ref{defNsol}, it
is natural to conjecture the following generalization of Theorems
\ref{1solthm} and \ref{2solthm}:   if $u \in H^{2N-2}$ is a
nontrivial distribution solution of the stationary equation
\begin{equation*}
d_3R_3 + d_5 R_5 + d_7 R_7 + \dots + d_{2N+3}R_{2N+3}=0,
\end{equation*}
then $u$ must be a $k$-soliton profile for the KdV hierarchy,
 for some $k \in \{1,2,\dots,N\}$.  More precisely, there must
 exist real numbers $\alpha_j$ and $a_j$, with $(-1)^{j-1} a_j > 0$ for
 $j=1,\dots,k$, such that
\begin{equation*}
u(x)=\psi^{(k)}(x;a_1,\dots,a_k;\alpha_1,\dots,\alpha_k).
\end{equation*}



Much of the proof given above for Theorem 5.2 generalizes
immediately to arbitrary $N$.   The extension of Lemma
\ref{regularity2} to arbitrary $N$, with $H^2$ replaced by
$H^{2N-2}$, is obvious.  Let $C_1,\dots C_N$ be the roots of the
equation
\begin{equation}
d_{2N+3}z^N - d_{2N+1}z^{N-1} + d_{2N-1}z^{N-2} - \dots  \pm d_3,
\end{equation}
and let $\alpha_j$ be the square roots of the $C_j$, suitably
defined.  The generalization to arbitary $N$ of the definition of
$\hat R(x,\zeta)$ is already given in \cite{Di}, along with the
proof that in the case when $\hat R(x,\zeta)$  has distinct roots
$\zeta_1,\dots,\zeta_N$ at some $x_0$, they satisfy an analogue of
the system \eqref{sys2}.

Using induction and an argument like that given above in Section
\ref{sec:2solconverse}, we can assume that the $\zeta_j$ are in
fact distinct, since otherwise \eqref{sys2} reduces to a system
with a smaller value of $N$. From \eqref{sys2} one then obtains a
generalization of the system \eqref{visyst} for functions $v_j$
which are suitably defined square roots of the functions
$-\zeta_j$.  Lemma \ref{impfnv}, Corollary \ref{ualphaandbeta},
and Lemma \ref{gensol} all generalize straightforwardly to
arbitrary $N$. We thus obtain that any solution $u(x)$ of
\eqref{stathier} is given by
\begin{equation}
u(x)=u(\alpha_1,\dots,\alpha_n,x)=\lim_{\beta_1 \to
\alpha_1,\dots,\beta_N \to
\alpha_N}\psi^{(N)}(x;a_1,\dots,a_N;\beta_1,\dots,\beta_N),
\label{ugenNaslim}
\end{equation}
for some complex numbers $a_1,\dots,a_N$, where the limit is taken
through values of $(\beta_1,\dots,\beta_N)$ such that the
$\beta_j$ are distinct and all non-zero.

To complete the proof of the conjectured general result, then, it
would remain to do two things.  First, establish an analogue of
Lemma \ref{uis2sol}, or in other words establish the conjecture
mentioned in the first remark after Definition \ref{defNsol}; and
second, show that in the degenerate cases when some of the $C_j$
coincide or are equal to zero, the functions
$\psi^{(N)}(x;a_1,\dots,a_N;\beta_1,\dots,\beta_N)$ in
\eqref{ugenNaslim} converge to a limit which cannot be
analytically continued to all of $\mathbf R$.

\section{Acknowledgements} \label{sec:ack}

We are deeply indebted to the late Leonid Dickey for the many hours he spent introducing us to his beautiful approach to soliton theory.  We would also like to thank Bernard Deconinck and an anonymous referee for helpful comments.

\end{document}